\documentclass[11pt]{amsart}
\usepackage{geometry}                % See geometry.pdf to learn the layout options. There are lots.
\usepackage{graphicx}
\usepackage{amssymb}
\usepackage{epstopdf}
\newtheorem{theorem}{Theorem}[section]
\newtheorem{prop}[theorem]{Proposition}

\theoremstyle{definition}
\newtheorem{remark}[theorem]{Remark}

\numberwithin{equation}{section}

\def\multi#1{\vbox{\baselineskip=0pt\halign{\hfil$\scriptstyle\vphantom{(_)}##$\hfil\cr#1\crcr}}}

\def \CAP{{\mathcal PF}}

\def\tttt #1{{\textstyle{#1} }}

\def \mur {\leftarrow}

\def \sch {{\cup\!\cup}}

\def \La {{\Lambda}}

\def \page{{\vfill\supereject}}

\def\la{{\lambda}}

\def \CD {{\mathcal D}}

\def \con {\subseteq}

\def\sig{\sigma}

\def\PI{\Pi}
\def\LL{\big\langle}

\def\RR {\big\rangle}

\def\DD {\Delta}

\def\OM {\Omega}

\def\om {\omega}
\def\la {\lambda}
\def \RA {\rightarrow}
\def \LA {\lefttarrow}

\def \sas {\vskip .06truein}
\def\sa{{\vskip .125truein}}

\def\sapp {{\vskip .5truein}}

\def \eee {\epsilon}

\def\aaa {\alpha}
\def\bbb {\beta}

\def\aa {\alpha}
\def\bb {\beta}

\def\con {\subseteq}
\def \ses {\enskip = \enskip}
\def \sps {\, + \,}

\def \sms {\, - \,}

\def \scs {\, , \,}
\def \ess {\enskip}

\def \ssp {\hskip .25em}
\def \bigsp {\hskip .5truein}
\def \part {\vdash}

\def \DD {\Delta}

\def \OM {{\Omega}}

\def \RA {{ \rightarrow }}
\def \LA {{ \leftarrow }}

\def \om {\omega}

\def \TK {{\tilde K}}
\def \TH {{\tilde H}}

\def \om {\omega}
\def \TK {{\tilde K}}
\def \TH {{\tilde H}}

\def \scs {\ssp , \ssp}
\def \ess {\enskip}
\def \ssp {\hskip .25em}
\def \bigsp {\hskip .5truein}
\def \part {\vdash}

\def \EE {{\mathcal E}}

\def \BB {{\bf B}}
\def \BC {{\bf C}}

\def \CPF {{\mathcal PF}_n}
\def \TS {\textstyle}

\title{A new ``dinv'' arising from the two part case of the Shuffle Conjecture}
\author{A. Duane, A. M. Garsia and M. Zabrocki}
\begin{document}

\begin{abstract}
In a recent paper \cite{Hag} J. Haglund showed that the expression
$\LL \Delta_{h_j} E_{n,k} , e_n\RR$ with $\Delta_{h_j}$ the Macdonald
eigen-operator
$
  \DD_{h_j} \TH_\mu = h_j[B_\mu] \TH_\mu
$ 
enumerates by $t^{ area} q^{dinv} $ the parking functions 
whose diagonal word is in  the shuffle $12\cdots j \sch j+1\cdots j+n$ 
with $k$ of the cars \hbox{$j+1,\ldots ,j+n$} in the
main diagonal including  car $j+n$ in the cell $(1,1)$. In view of some recent
conjectures of Haglund-Morse-Zabrocki \cite{HMZ} it is natural to conjecture  that 
replacing $E_{n,k}$ by the modified Hall-Littlewood funtions
$\BC_{p_1}\BC_{p_2}\cdots \BC_{p_k}\, 1$ 
would yield a polynomial that enumerates the same collection of parking functions but now restricted by the requirement   that 
the Dyck path supporting cars $j+1,\ldots ,j+n$ hits the diagonal 
according to the composition $p=(p_1,p_2,\ldots ,p_k)$. We prove here this conjecture by deriving a recursion for the polynomial 
$\LL \Delta_{h_j} \BC_{p_1}\BC_{p_2}\cdots \BC_{p_k}\, 1 \scs e_n\RR$ then
using this recursion to construct a   new dinv  statistic we will denote $ndinv$
and show that this polynomial enumerates the latter
parking functions by  $ t^{ area} q^{ndinv}$

\end{abstract}

\maketitle
%\section{}
%\subsection{}

\begin{section}{Introduction}

Parking functions are endowed by a colorful history and   jargon (see for instance \cite{Hagbook})
that is very helpful in dealing with them combinatorially as well as analytically.
Here we will represent them interchangeably  as two line arrays or as 
tableaux. A single example of this correspondence should  be sufficient for our purposes. In the figure below we have on the left the  
two line array, with the list of cars $V=(v_1,v_2,\ldots ,v_n)$ 
on top and their  diagonal numbers  $U=(u_1,u_2,\ldots ,u_n)$
on the bottom. In the corresponding $n\times n$ tableau of lattice cells
we have shaded the {\it main diagonal } (or $0$-diagonal) and 
drawn the {\it supporting}
Dyck path. The component  $u_i$ gives the number of lattice cells EAST of the $i^{th}$
NORTH step  and WEST of the main diagonal. 
The cells adjacent to the NORTH
steps of the path are filled with the corresponding cars  from bottom to top. 
\vskip -.3in

\begin{equation}
PF=\begin{bmatrix}
4 &  6 &  8 &  1 &  3  & 2 &  7 &  5 \\
0 &  1 &  2 &  2 &  3  & 0 &  1 &  1 \\
\end{bmatrix}
\ess\ess\ess\Longleftrightarrow\ess\ess\ess
 \vcenter{\hbox{\includegraphics[width=1.5in]{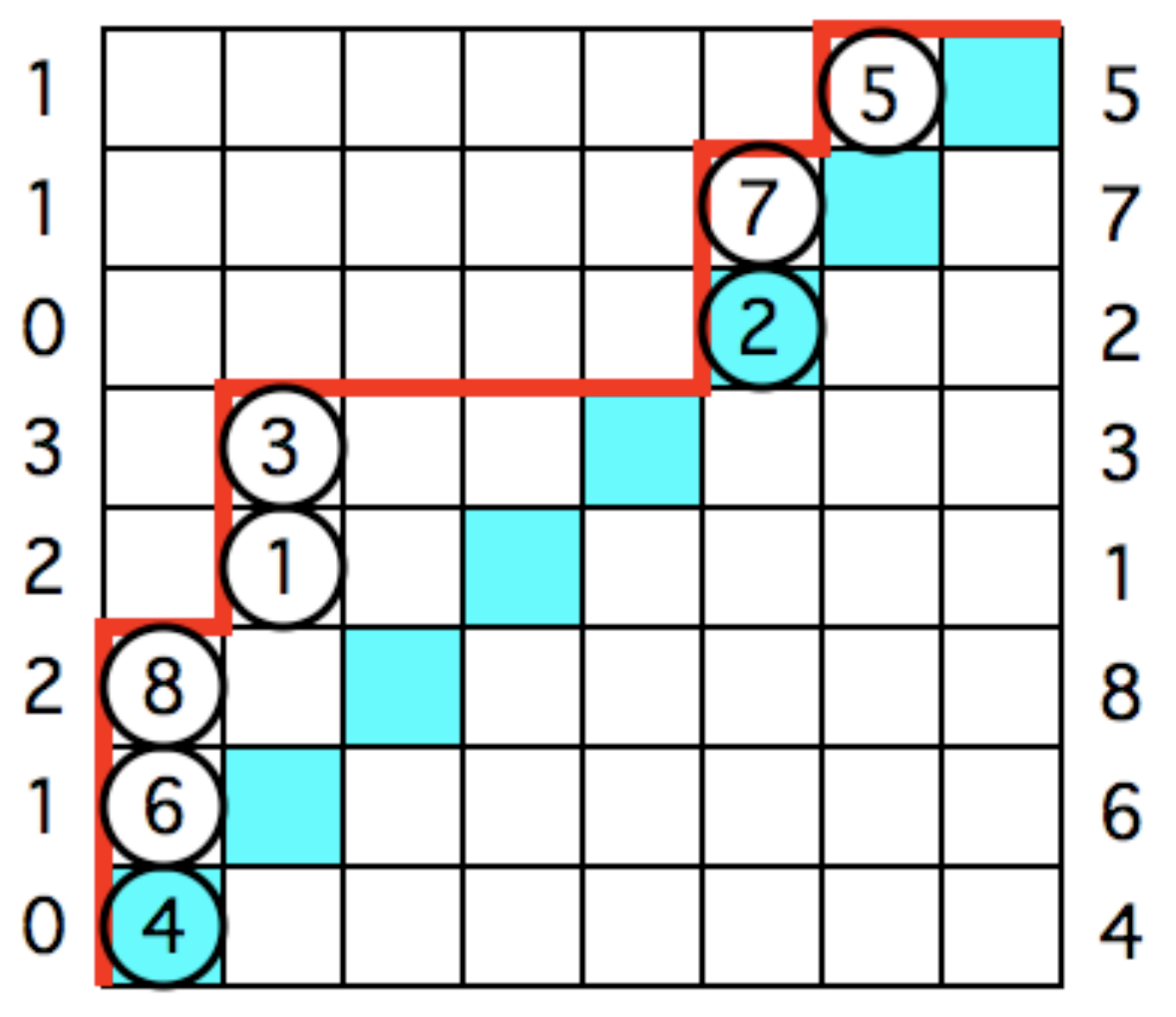}}}
%\eqno \II.1
\label{eqIIp1}
\end{equation}
The resulting tableau uniquely represents a parking function if and only if
the cars increase up the columns.

A necessary and sufficient condition for the vector U to give a Dyck path
is that 
\begin{equation}
u_1=0\ess\ess\ess \hbox{and }\ess\ess\ess\ess 0\le u_i\le u_{i-1}+1
%\eqno\II.2
\label{eqIIp2}
\end{equation}
This given, the column increasing property of the corresponding  tableau 
is assured  by the 
requirement that  $V=(v_1 , v_2 , \ldots  ,  v_n)$ is a permutation in $S_n$ 
satisfying
\begin{equation}
u_i=u_{i-1}+1\ess\Longrightarrow\ess v_i>v_{i-1}  
%\eqno\II .3
\label{eqIIp3}
\end{equation}
 
\noindent
We should   mention that  the component $u_i$ may also be viewed
as the order of the diagonal supporting car $V_i$.  
In the example above, car $3$ is in the third diagonal,   $1$ and $8$ are in the second 
diagonal,   $5,7$ and $6$ are in the first diagonal and $2$ and $4$ are in 
the main diagonal. We have purposely listed the cars by diagonals from right to left
starting with the highest diagonal. This gives the {\it diagonal  word} 
of $PF$  which we will denote $\sig(PF)$.
It is easily seen that $\sig(PF)$
can also be obtained directly  from the $2$-line array 
by successive right to left readings of the components 
of the vector $V=(v_1 , v_2 , \ldots  ,  v_n)$ according to decreasing values of
$u_1 , u_2 , \ldots  ,  u_n$. In previous work, each parking function is assigned a {\it weight }
\begin{equation}
w(PF)\ses t^{area(PF)} q^{dinv(PF)} 
%\eqno\II .4
\label{eqIIp4}
\end{equation}
where  
\begin{equation}
area(PF)=u_1+u_2+\cdots +u_n
\label{eqIIp5}
%\eqno \II.5
\end{equation}
and
\begin{equation} \label{eqIIp6}
dinv(PF)= 
\hskip -.1in\sum_{1\le i<j\le n }\Big( \chi( u_i=u_j\, \&\, v_i<v_j) 
\sps  \chi( u_i=u_j+1\, \&\, v_i>v_j)\Big)
%\eqno \II.6
\end{equation}
It is clear from this imagery, 
that the sum in \eqref{eqIIp5} gives the total number of cells between the supporting Dyck path and 
the main diagonal. We also see that two cars in the same diagonal with the car on the left smaller 
than the car on the right will contribute a unit to $dinv(PF)$.
The same holds true when a car on the left
is bigger than a car on the right with the latter in the adjacent lower diagonal.
Thus in the  the present example we have  
$$
 area(PF)=10,\ess dinv(PF)=4,\ess \sig(PF)=31857624,
$$
yielding
$$
w(PF)\ses t^{10}q^4  
$$
Here and after, the vectors $U$ and $V$ 
in the two line representation will be  also referred to as  $U(PF)$ and $V(PF)$.
It will also be convenient to denote by $\CPF$ the collection of parking functions
in the $n\times n$ lattice square. 
\sas 

The {\it Shuffle conjecture } \cite{HHLRU} states that for any partition 
$\mu=(\mu_1,\mu_2,\ldots ,\mu_\ell)\part n$
we have the identity
\begin{equation}
\LL\nabla e_n\scs h_{\mu_1}h_{\mu_2}\cdots h_{\mu_\ell}\RR =\sum_{PF\in \CPF }
t^{area(PF) }q^{dinv(PF)}\chi(\sig(PF)\in 
\EE_1 \sch \EE_2 \sch \cdots \sch \EE_\ell)
%\eqno \II.7
\label{eqIIp7}
\end{equation}
where $\nabla$ is the Macdonald eigen-operator introduced in \cite{BG},
$e_n$ is the familiar elementary symmetric function,
$ h_{\mu_1}h_{\mu_2}\cdots h_{\mu_\ell}$ is the {\it homogeneous}
symmetric function basis indexed by $\mu$, 
$\EE_1,\EE_2,\ldots ,\EE_\ell$ are successive segments 
of the word $1234\cdots n$ of respective lengths
$\mu_1,\mu_2,\ldots ,\mu_\ell$ and the symbol 
$\chi(\sig(PF)\in 
\EE_1\sch \EE_2\sch \cdots \sch \EE_\ell)$ is to indicate that the sum is to be carried out over  parking functions in $\CPF$ whose diagonal word
is a shuffle of the words $\EE_1,\EE_2,\ldots ,\EE_\ell$.
In \cite{Hag} Haglund proved the $l=2$ case of \eqref{eqIIp7}. By a remarkable 
sequence of identities it is shown in \cite{Hag} that this case is a consequence of the 
more refined identity
\begin{equation}
\LL \Delta_{h_J} E_{n,k} \scs e_n \RR\ses\hskip -.2in \sum_{PF \in\CAP_{n+J}(k)}
t^{area(PF)}q^{dinv(PF)}\chi(\sig(PF)\in \EE_J \sch \EE_{n-J})
\ess\ess\ess
\hbox{for $1\le k\le n$}
%\eqno \II.8
\label{eqIIp8}
\end{equation}
with $\EE_J=12\cdots J$, $\EE_{n-J}=J+1\cdots J+n$, and the sum is over
the collection 
 $\CAP_{n+J}(k)$ of parking functions in the $n+J\times n+J$
lattice square that have $k$ of the cars \hbox{$J+1,\ldots ,J+n$} in the
main diagonal including  car $J+n$ in the cell $(1,1)$. 
Here the $E_{n,k}$ are certain ubiquitous symmetric
functions introduced in \cite{GHag}  with sum
\begin{equation}
E_{n,1}+E_{n,2}+\cdots + E_{n,k}=e_n
%\eqno \II.9
\label{eqIIp9}
\end{equation}
and $\Delta_{h_j}$ is the linear operator obtained by setting
for the modified Macdonald basis in \cite{Mac}, \cite{GHai1}.
\begin{equation}
\Delta_{h_J}  \TH_\mu[X;q,t] \ses 
h_J\Big[\sum_{(i,j)\in \mu}t^{i-1}q^{j-1}\Big]\TH_\mu[X;q,t]
%\eqno \II.10
\label{eqIIp10}
\end{equation}
More recently, J. Haglund, J. Morse and M. Zabrocki \cite{HMZ} formulated 
a variety of new conjectures yielding surprising refinements of the shuffle conjecture.
In \cite{HMZ} they introduce a new ingredients in the Theory of parking functions. This is the 
{\it diagonal composition} of a Parking function,
which we denote by $p(PF)$  and  is simply the composition which
gives the position of the  zeros in the vector  $U=(u_1,u_2,\ldots ,u_n)$,
or equivalently the lengths of the segments of the main diagonal between\
successive hits of its supporting Dyck path.
One of their conjectures is the identity
\begin{align}
\LL\nabla \BC_{p_1}\BC_{p_2}\cdots \BC_{p_k}\, 1
\scs & h_{\mu_1},h_{\mu_2},\cdots,h_{\mu_\ell}\RR = \nonumber\\
&=\hskip -.2in
\sum_{\multi{PF\in \CPF \cr p(PF)=(p_1,p_2,\ldots ,p_k)}}
\hskip -.2in t^{area(PF)}
q^{dinv(PF)}\chi(\sig(PF)\in \EE_1\sch \EE_2\sch \cdots \sch \EE_\ell) 
%\eqno \II.11
\label{eqIIp11}
\end{align}
being valid for all $p\models n$ and 
$\mu=\part n$. Where
for each integer $a$ , $\BC_a$ is the operator
whose action on a  symmetric function $F[X]$, in plethystic notation,
can be simply expressed in the form
\begin{equation}
\BC_a F[X]\ses (-{\TS{ 1\over q }})^{a-1} F\big[X-\TS{1-1/q \over z}
\big]\displaystyle
\sum_{m\ge 0}z^m h_{m}[X]\, \Big|_{z^a},
%\eqno \II.12
\label{eqIIp12}
\end{equation}
Remarkably,  the operators in  \eqref{eqIIp11} appear to control the shape
of the supporting Dyck paths. Since in \cite{HMZ} it is shown that we also have
the identity
\begin{equation}
E_{n,k}\ses \sum_{p_1+p_2+\cdots +p_k=n}  \BC_{p_1}\BC_{p_2}\cdots \BC_{p_k}\, 1 
%\eqno \II.13
\label{eqIIp13}
\end{equation}
 it comes natural to inquire what becomes of Haglund's identity \eqref{eqIIp8}
 when $E_{n,k}$ is replaced by one of the symmetric polynomials  $\BC_{p_1}\BC_{p_2}\cdots \BC_{p_k}\, 1$ . Note, however that since the $k$ in \eqref{eqIIp8},
 under the action of $\DD_{h_J}$  
controls the number of {\it big cars} on the main diagonal, it natural to suspect
that the combination of $\DD_{h_J}$ and  $\BC_{p_1}\BC_{p_2}\cdots \BC_{p_k}\, 1$
would result in forcing $k$ of the big cars to hit the diagonal according
to the composition $p=(p_1,p_2,\ldots ,p_k)$. Miraculous as this might appear to be,
computer data beautifully confirms this mechanism $\ldots$ but up to a
point. In fact, following this line of reasoning,   one might conjecture 
the identity
\begin{align}
\LL \Delta_{h_J}& \BC_{p_1}\BC_{p_2}\cdots \BC_{p_k}\, 1 \scs e_n \RR
\ses \nonumber\\
& \sum_{\multi{PF\in \CAP_{J+n}(k)
\cr p(big(PF))=(p_1,p_2,\ldots ,p_k)}}\hskip -.3in
t^{area(PF)}q^{dinv(PF)}\chi(\sig(PF)\in 12\cdots J\sch J+1\cdots J+n)
%\eqno \II.14
\label{eqIIp14}
\end{align}
where $p(big(PF))$ now refers to the diagonal composition of the big cars, but otherwise the sum is over the same parking functions 
occurring in \eqref{eqIIp8}. Now that turned out to be false. Yet
computer data revealed that  the following ($q$-reduced) 
version of \eqref{eqIIp8} is actually true. Namely
$$
\LL \Delta_{h_J} \BC_{p_1}\BC_{p_2}\cdots \BC_{p_k}\, 1 \scs e_n \RR
\Big|_{q\RA 1}
\ses \hskip -.3in \sum_{\multi{PF\in \CAP_{n+J}(k)\cr p(big(PF))=(p_1,p_2,\ldots ,p_k)}}\hskip -.3in
t^{area(PF)}\chi(\sig(PF)\in 12\cdots j\sch j+1\cdots j+n)
$$
This circumstance led to the conjecture that \eqref{eqIIp14}  could be made true
by replacing the classical parking function ``{\it dinv}'' by a new dinv  
more focused on the positions of the big cars.
\sa

The main result of this paper is a proof of this conjecture. Banking on the intuition gained from previous work \cite{GXZ}  and using some of the identities
developed there with the $\BC_a$ and  $\BB_b$ operators we are able to derive the following basic recursion.
\sas

%Theorem I.1
\begin{theorem}\label{th:IIp1}
For all compositions $p=(p_1,p_2,\ldots ,p_k)$  we have
\begin{align}
\LL \Delta_{h_j}\BC_{p_1}\BC_{p_2}\cdots \BC_{p_k} 1\scs  e_n \RR 
&\ses
t^{p_1-1}q^{k-1}
\LL \Delta_{h_{j-1}} \BC_{p_2}\cdots \BC_{p_k}\BB_{p_1} 1\scs  
 e_{n} \RR \nonumber\\
 &\bigsp\bigsp
 \sps \chi(p_1=1)\LL \Delta_{h_{j}}\BC_{p_2}\cdots \BC_{p_k} 1\scs   e_{n-1} \RR
 \label{eqIIp15}
%\eqno \II.15 
\end{align}
 with $\BB_a=\om \widetilde \BB_a\om$ and for any symmetric function $F[X]$ 
\begin{equation}
\widetilde\BB_a  F[X]\ses F\big[ X-  \TS{1-q \over z}\big]\OM[ 
 zX]\, \Big|_{z^a}
\label{eqIIp16}
%\eqno \II.16
\end{equation}
\end{theorem}
Now the Haglund-Morse-Zabrocki conjectures also assert that, replacing the 
$\BC$ operators by the $\BB$ operators in \eqref{eqIIp11}
has the effect of allowing 
the controlled Dyck paths to hit the diagonal everywhere, including
the points forced by the composition  $p$. This led us to interpret 
the first polynomial on the right hand side of \eqref{eqIIp8} as a  weighted enumeration
of the  collection of parking functions with diagonal word of 
a shuffle of $12\cdots (j-1)$ by $j\cdots n+j-1$ whose big cars
hit the main diagonal according to the collection of compositions obtained by 
concatenating $(p_2,\ldots,p_k)$ with an arbitrary composition of $p_1$.
Guided by this interpretation we succeeded to obtain by means of \eqref{eqIIp16}
a recursive construction of the appropriate   new dinv  and prove the
identity
\begin{align}
\LL \Delta_{h_j} &\BC_{p_1}\BC_{p_2}\cdots \BC_{p_k}\, 1 \scs e_n \RR
\ses  \nonumber\\
&\sum_{\multi{PF\in \CAP_{J+n}(k) 
\cr p(big(PF))=(p_1,p_2,\ldots ,p_k)}}\hskip -.3in
t^{area(PF)}q^{ndinv(PF)}\chi(\sig(PF)\in 12\cdots J\sch J+1\cdots J+n)
\label{eqIIp17}
%\eqno \II.17
\end{align}
To carry out all this we need a collection of identities of Macdonald polynomial theory
already used in previous work. These identities and the corresponding notational 
conventions will be collected  in the first section with references  to 
the original sources for their proofs. The second section will be dedicated to 
the proof of Theorem \ref{th:IIp1}. All the corresponding combinatorial reasoning
including the construction of the new dinv is  given in the third section where
our ``{\it ndinv}'' is also given an equivalent somewhat less recursive construction
with the hope that it may be conducive  to the discovery of a direct
formula  for the new dinv which, as in the case of the classical dinv, 
is closely related to the geometry of the corresponding parking function
diagram. 

\end{section}
\vfill
\newpage

\begin{section}{Auxiliary identities from the Theory of Macdonald polynomials}

The space of symmetric polynomials will be denoted $\Lambda$. The subspace of
homogeneous symmetric polynomials of degree $m$ will be denoted by
$\Lambda^{=m}$.
We will seldom work with symmetric polynomials expressed in terms of variables but rather express them
in terms of one of the six classical symmetric function bases
\sas
\noindent
{(1)} ``{\it power}''   $\{p_\mu\}_\mu$, 
{ (2)}   ``{\it monomial}''  $\{m_\mu\}_\mu$, 
{ (3)}  ``{\it homogeneous}''  $\{h_\mu\}_\mu$, \\
{ (4)}   ``{\it elementary}''  $\{e_\mu\}_\mu$,
{ (5)}   ``{\it forgotten}''  $\{f_\mu\}_\mu$  
and
{ (6)}   ``{\it Schur}''  $\{s_\mu\}_\mu$~.

We recall that the fundamental involution $\om$ may be defined by setting 
for the power basis indexed by  $\mu=(\mu_1,\mu_2,\ldots ,\mu_k)\part n$

\begin{equation}
\om p_\mu\ses (-1)^{n-k}p_\mu\ses  (-1)^{|\mu|-l(\mu) }p_\mu
\label{eq1p1}
%\eqno 1.1
\end{equation}
where for any vector $\ess v=(v_1,v_2,\cdots,v_k)$ we set
 $
\ess |v|=\sum_{i=1}^k v_i 
$
and 
$ l(v)=k $.

In dealing with symmetric function identities, specially with those arising
in the Theory of Macdonald Polynomials, we find it convenient
and often indispensable to use 
plethystic notation. This device has a straightforward definition which can be
verbatim implemented in {\it MAPLE} or {\it MATHEMATICA}
for computer experimentation. We simply set for
any expression $E=E(t_1,t_2 ,\ldots )$ and any power symmetric function $p_k$
\begin{equation}
p_k[E]\ses E(\, t_1^k,t_2^k,\ldots ).
\label{eq1p2}
%\eqno 1.2
\end{equation}
This given, for any symmetric function $F$ we set
\begin{equation}
F[E]\ses Q_F(p_1,p_2, \ldots )\Big|_{p_k\RA E(\, t_1^k,t_2^k,\ldots )}
%\eqno 1.3
\label{eq1p3}
\end{equation}
where $Q_F$ is the polynomial yielding the expansion of $F$ in 
terms of the power basis. Note
that in writing  $E(t_1,t_2,\ldots )$ we are tacitly assuming 
that $ t_1,t_2,t_3,\ldots $
are all the variables appearing in $E$ and in writing
 $ E(t_1^k,t_2^k,\ldots )$ we intend
that all the variables appearing in $E$ have been raised 
to their $k^{th}$ power.

A paradoxical but necessary property of plethystic substitutions is 
 that \eqref{eq1p2} requires  
\begin{equation}
p_k[-E]\ses -p_k[E].
%\eqno 1.5
\label{eq1p5}
\end{equation}

This notwithstanding, we will still need to carry out ordinary  changes 
of signs. To distinguish it from the {\it plethystic} minus sign,
we will carry out the {\it ordinary} sign change  by
 prepending our expressions
with a superscripted minus sign, or as the case may be, by means
of a new variables $\eee$ 
which outside of the plethystic bracket  is simply replaced by $-1$. 
For instance, these conventions give for $X_k=x_1+x_2+\cdots +x_n$
$$
 p_k[- ^-X _n]\ses   (-1)^{k-1} \sum_{i=1}^nx_i^k
$$
or, equivalently
$$
p_k[ -\eee X_n]\ses  -\eee^k\sum_{i=1}^nx_i^k\ses  (-1)^{k-1} \sum_{i=1}^nx_i^k
$$  
In particular we get for $X=x_1+x_2+x_3+\cdots $ 
$$
 \om p_k[X]\ses p_k[-^-X ]  .
$$
Thus  for any symmetric function $F\in \Lambda$ and any expression $E$ we have
\begin{equation}
\om F[E]\ses F[-^-E]\ses F[-\eee E]
%\eqno 1.6
\label{eq1p6}
\end{equation}

In particular,  if $F\in \Lambda^{=k} $  we may also rewrite this as
\begin{equation}
F[-E]\ses   \om F[ ^-E]\ses (-1)^k \om F[ E].
%\eqno 1.7
\label{eq1p7}
\end{equation}

The formal power series
\begin{equation}
\OM\ses exp\Big(\sum_{k\ge 1}{p_k\over k}\Big)
%\eqno  1.8
\label{eq1p8}
\end{equation}
 combined with plethysic substitutions will  provide   
a powerful way of dealing with the many generating functions occurring in
our manipulations.

Here and after it will be convenient to 
identify partitions with their (french) Ferrers diagram. Given
a partition $\mu$ and a cell $c\in \mu$, Macdonald introduces four parameters
$l =l_\mu(c)$, $l'=l'_\mu(c)$, $a =a_\mu(c)$ and  $a'=a'_\mu(c)$ 
called 
{\it leg, coleg, arm } and {\it coarm } which  give the 
number of lattice cells of $\mu$  strictly  NORTH,  
SOUTH, 
EAST and WEST of $c$, (see attached figure).
Following Macdonald we will set
\begin{equation}
\bigsp n(\mu)\ses \sum_{c\in\mu} l_\mu(c)\ses \sum_{c\in\mu} l'_\mu(c)
\ses \sum_{i=1}^{l(\mu)} (i-1)\mu_i. 
%  1.9
\label{eq1p9}
\end{equation}

%\vskip-2in
%\hsize 6.5in

$${\includegraphics[width=1.8in]{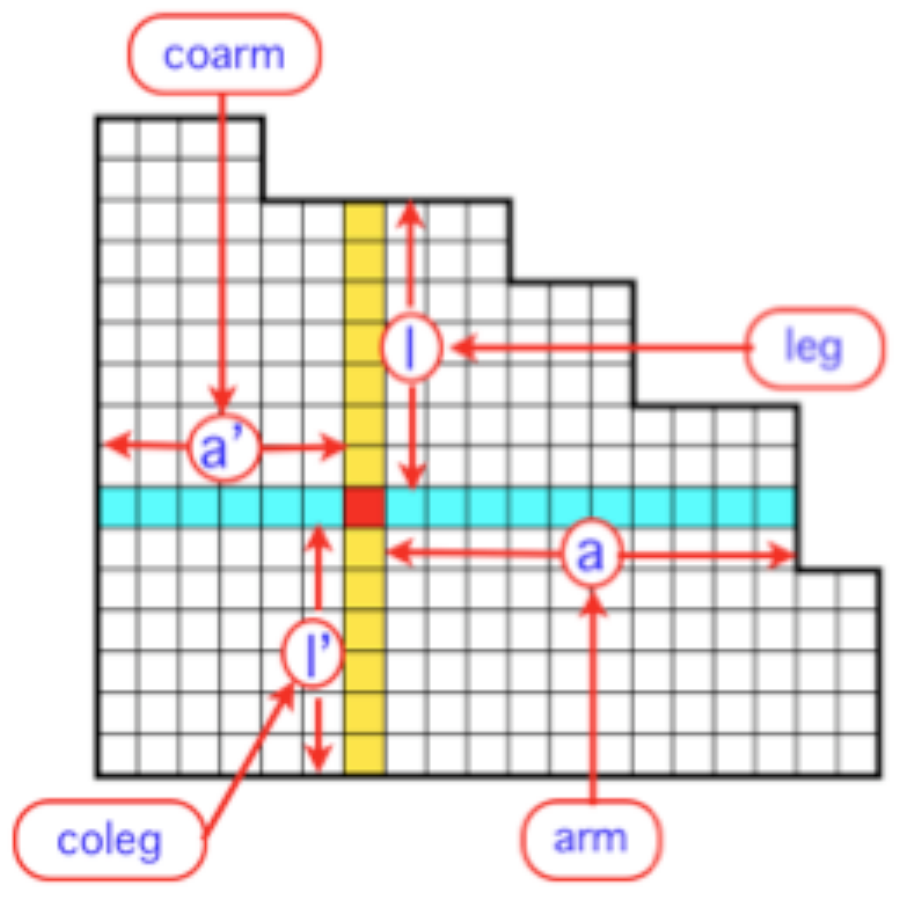}} $$ 

Denoting by $\mu'$ the conjugate of $\mu$, 
the basic ingredients playing a role in the theory of Macdonald
polynomials are  

\begin{align}
  T_\mu=&t^{n(\mu)}q^{n(\mu')}
\scs  
\ess\ess
B_\mu(q,t)  =\sum_{c\in \mu}t^{l'_\mu(c)}q^{a'_\mu(c)}\scs \nonumber\\
 \Pi_\mu(q,t)&= \prod_{c\in \mu;c\neq(0,0)}
(1-t^{l'_\mu(c)}q^{a'_\mu(c)}) \scs\ess\ess M=(1-t)(1-q)\label{eq1p10}
\\
D_\mu(q,t)= MB_\mu(q,t)- 1& \scs\ess\ess
 w_\mu(q,t)=\prod_{c\in \mu}(q^{a _\mu(c)} -t^{l _\mu(c)+1})(t^{l _\mu(c)} -q^{a _\mu(c)+1}),
\nonumber
%\eqno 1.10
\end{align}
together with a deformation of the Hall scalar product, which we call
the {\it star} scalar product,   defined by setting for the power basis 
$$
\LL p_\la\scs p_\mu \RR_*\ses 
(-1)^{|\mu|-l(\mu)} \prod_i (1-t^{\mu_i})(1-q^{\mu_i})\ssp z_\mu\ssp
\chi(\la=\mu) ,
$$
where $z_\mu$ gives the order of the stabilizer of a permutation with cycle structure $\mu$.
\sa

This given, the modified Macdonald Polynomials we will deal with here are the unique 
symmetric function basis $\big\{\TH_\mu(X;q,t)\big \}_\mu$ which  is
upper triangularly related to the basis $\{s_\la[\tttt{X\over t-1}]\}_\la$
and satisfies the orthogonality condition
\begin{equation}
\LL \TH_\la\scs \TH_\mu\RR_*\ses \chi(\la=\mu) w_\mu(q,t)
%\eqno 1.11
\label{eq1p11}
\end{equation}

In this writing we will make intensive use of the operator $\nabla$ defined by setting for all partitions $\mu$
$$
\nabla \TH_\mu \ses T_\mu \TH_\mu.
$$
A closely related family of symmetric function operators is obtained
by setting for a symmetric function $F[X]$ 
$$
\DD_F \TH_\mu \ses F[B_\mu] \TH_\mu
$$
It is good to keep in mind that, because of the relation $e_n[B_\mu]= T_\mu$ for $\mu\part n$, the operator 
$\nabla$ itself reduces to $\DD_{e_n}$ when acting  
symmetric polynomials that are  homogeneous of degree $n$.

Recall that for our version of the Macdonald polynomials
the Macdonald Reciprocity formula states that
\begin{equation}
{\TH_\aaa[1+u\, D_\bbb ] \over \prod_{c\in \aaa }
(1-u\, t^{l'}q^{a' })}
\ses 
{\TH_\bbb[1+u\,D_\aaa ] \over \prod_{c\in \bbb }
(1-u\, t^{l'}q^{a' })}
\ess\ess\ess\ess\ess\ess\hbox {(for  all   pairs $\aaa,\bbb$)}
%\eqno 1.12
\label{eq1p12}
\end{equation}

 We will use here several  special evaluations of \eqref{eq1p12}. To begin, canceling the 
common factor $(1-u)$ out of  the denominators on both sides of 
\eqref{eq1p12} then setting $u=1$ gives
\begin{equation}
{\TH_\aaa[M B_\bbb ] \over  \Pi_\aa}
\ses 
{\TH_\bbb[ M B_\aaa ] \over\Pi_\bb}
\ess\ess\ess\ess\ess\ess\hbox {(for  all   pairs $\aaa,\bbb$)}
%\eqno 1.13
\label{eq1p13}
\end{equation}

On the other hand replacing $u$ by $1/u$ and letting $u=0$ in \eqref{eq1p12}
gives
\begin{equation}
(-1)^{|\aaa|}{\TH_\aaa[D_\bbb ] \over   T_\aaa}
\ses 
(-1)^{|\bbb|}{\TH_\bbb[ D_\aaa ] \over T_\bb}
\ess\ess\ess\ess\ess\ess\hbox {(for  all   pairs $\aaa,\bbb$)}
%\eqno 1.14
\label{eq1p14}
\end{equation}

Since for $\bbb$ the empty partition we can take $\TH_\bbb=1$ and $D_\bbb=-1$,
\eqref{eq1p12}  in this case  for $\aa=\mu$ reduces to
\begin{equation}
 \TH_\mu[1-u\, ]\ses \prod_{c\in \mu}(1-ut^{l'}q^{a'})\ses(1-u)\sum_{r=0}^{n-1}
(-u)^r e_r[B_\mu-1]
%\eqno 1.15
\label{eq1p15}
\end{equation}
This identity yields the coefficients of hook Schur functions in the expansion.
\begin{equation}
\TH_\mu[X;q,t]\ses \sum_{\la\part |\mu|}s_\mu[X] \TK_{\la\mu}(q,t)
%\eqno 1.16
\label{eq1p16}
\end{equation}
Recall that the addition formula for Schur functions gives
\begin{equation}
s_\mu[1-u]\ses 
\begin{cases}
(-u)^r(1-u) & \hbox{ if }\mu=(n-r,1^r)\\
 0 & otherwise
\end{cases}
%\eqno 1.17
\label{eq1p17}
\end{equation}
Thus \eqref{eq1p16}, with $X=1-u$, combined with \eqref{eq1p15}  gives for $\mu\part n$
$$
\LL  \TH_\mu\scs s_{(n-r,1^r)}\RR\ses e_r[B_\mu-1]
$$
and the identity $e_rh_{n-r}=s_{(n-r,1^r)}+s_{(n-r-1,1^{r-1})}$ gives
\begin{equation}
\LL  \TH_\mu\scs e_rh_{n-r}\RR\ses e_r[B_\mu].
%\eqno 1.19
\label{eq1p19}
\end{equation}
Since for  $\bbb=(1)$ we have $\TH_\bbb=1$
and $\Pi_\bbb=1$, formula \eqref{eq1p13} reduces to the 
surprisingly simple identity
\begin{equation}
\TH_\aaa[M]\ses MB_\aaa \Pi_\aaa.
%\eqno 1.20
\label{eq1p20}
\end{equation}
Last but not least we must also recall that we have  
  the Pieri formulas
\begin{equation}
a)\ess\ess e_1\TH_\nu\ses \sum_{\mu\LA \nu}d_{\mu\nu}\TH_\mu\scs
\ess\ess\ess\ess\ess\ess\ess\ess\ess
b)\ess\ess e_1^\perp\TH_\mu\ses \sum_{\nu\RA \mu}c_{\mu\nu}\TH_\nu\scs
%\eqno 1.21
\label{eq1p21}
\end{equation}
and their corresponding summation formulas (see   \cite{BGHT},\cite{GHT},\cite{Zab} )
\begin{equation}
\sum_{\nu\RA\mu}c_{\mu\nu}(q,t)\, (T_\mu/T_\nu)^k\ses 
\begin{cases}
{{tq}\over M}\ssp h_{k+1}\big[D_\mu(q,t)/tq\big] &\hbox{ if }k \geq 1\\
B_\mu(q,t) & \hbox{ if }k=0 
\end{cases}
%\eqno 1.22
\label{eq1p22}
\end{equation}
\begin{equation}
\sum_{\mu\leftarrow\nu}d_{\mu\nu}(q,t)\, (T_\mu/T_\nu)^k\ses 
\begin{cases}
(-1)^{k-1}\ssp e_{k-1}\big[D_\nu(q,t) \big] & \hbox{ if } k\geq 1,\\
1 & \hbox{ if } k=0
\end{cases}
%\eqno 1.23
\label{eq1p23}
\end{equation}
Here $\nu\RA \mu$ simply means that the sum is over  $\nu$'s   obtained 
from $\mu$ by removing a corner cell and  $\mu\LA \nu$ 
means  that the sum is over $\mu$'s   obtained 
from $\nu$ by adding a corner cell.

It will also  be useful to know that these two Pieri coefficients are related by the identity
\begin{equation}
d_{\mu\nu}\ses M c_{\mu\nu} \tttt{w_\nu\over w_\mu} 
%\eqno 1.24
\label{eq1p24}
\end{equation}
Recall that the Hall scalar product in the theory of Symmetric functions
may be defined by setting, for the power basis
\begin{equation}
\LL p_\la\scs p_\mu \RR \ses 
 z_\mu\ssp
\chi(\la=\mu) ,
%\eqno 1.25
\label{eq1p25}
\end{equation}
It follows from this that
the $*$-scalar product, is simply related to the   
Hall scalar product
by setting for all pairs of symmetric functions $f,g$
\begin{equation}
\LL f\scs g\RR_* \ses \LL f\scs \om \phi g\RR,
%\eqno 1.26
\label{eq1p26}
\end{equation}
where it has been customary to let $\phi$ be the operator
defined by setting for any symmetric function $f$
\begin{equation}
\phi\, f[X]\ses f[MX].
%\eqno 1.27
\label{eq1p27}
\end{equation}
Note that the inverse of $\phi$ is usually written in the form
\begin{equation}
f^*[X]\ses f[X/M]
%\eqno 1.28
\label{eq1p28}
\end{equation}
In particular we also have  for all symmetric functions $f,g$
\begin{equation}
\LL f\scs g\RR \ses \LL f, \om g^* \RR_* ~.
%\eqno 1.29
\label{eq1p29}
\end{equation}
The orthogonality relations in \eqref{eq1p11} yield the Cauchy identity
for our Macdonald polynomials in the form
\begin{equation}
\OM\left[-\eee \tttt{XY\over M} \right]\ses \sum_{\mu}
{\TH_\mu[X]\TH_\mu[Y]\over w_\mu}
%\eqno 1.30
\label{eq1p30}
\end{equation}
which restricted to its homogeneous component of degree $n$ in $X$ and $Y$ 
reduces to
\begin{equation}
e_n\left[  \tttt{XY\over M} \right]\ses \sum_{\mu\part n}
{\TH_\mu[X]\TH_\mu[Y]\over w_\mu}~.
%\eqno 1.31
\label{eq1p31}
\end{equation}

Note that the orthogonality relations in \eqref{eq1p11} yield us 
the following  Macdonald polynomial expansions 
\sas

%Proposition 1.1
\begin{prop}\label{prop:1p1}
For all $n\ge 1$ we have
\begin{align}
a)&\ess  
e_n\big[\tttt{X\over M} \big]= \sum_{\mu\part n}
{\TH_\mu[X]\over w_\mu}\nonumber
\scs
\\   
b)&\ess 
h_k\big[\tttt{X\over M} \big]
e_{n-k}\big[\tttt{X\over M} \big]= \sum_{\mu\part n}
{e_k[B_\mu]\TH_\mu[X]\over w_\mu}~,\nonumber
\scs
\\  
c)&\ess  
h_n\big[\tttt{X\over M} \big]= \sum_{\mu\part n}
{T_\mu \TH_\mu[X]\over w_\mu}~,\nonumber
\\
d)&\ess (-1)^{n-1} p_n\ses (1- t^n)(1-q^n) \sum_{\mu\part n}
{\Pi_\mu \TH_\mu[X]\over w_\mu}~,\label{eq1p32}
\\
e)&\ess\ess\ess
e_1[X/M]^n 
\ses 
\sum_{\mu\part n}
{\TH_\mu[X]\over w_\mu} \LL \TH_\mu , e_1^n\RR~,\nonumber
\\
f)&\ess\ess\ess
e_n \ses
\sum_{\mu\part m}
{\TH_\mu[X] MB_\mu \Pi_\mu\over w_\mu}~.\nonumber
\end{align}
%\eqno 1.32
\end{prop}

Finally it is good to keep in mind, for future use, that we have for all partitions $\mu$   
\begin{equation}
T_\mu   \om \TH_\mu[X;1/q,1/t]\ses  \TH_\mu[X;q,t] 
%\eqno 1.33
\label{eq1p33}
\end{equation}

\begin{remark} 
It was conjectured in \cite{GHai2} and proved in \cite{Hai} that the  bigraded Frobenius characteristic
of the diagonal Harmonics of $S_n$ is given by the symmetric function
\begin{equation}
DH_n[X;q,t]\ses \sum_{\mu\part n}{T_\mu  \TH_\mu(X;q,t)M B_\mu(q,t) 
\Pi_\mu(q,t) \over w_\mu(q,t)}
%\eqno 1.34
\label{eq1p34}
\end{equation}
Surprisingly the intricate rational function on the right hand 
side is none other than $\nabla e_n$. To see this we simply combine 
the relation in \eqref{eq1p20} with the degree $n$ restricted Macdonald-Cauchy formula
\eqref{eq1p31} obtaining
\begin{equation}
e_n[X]=e_n\left[\tttt{XM\over M} \right]\ses \sum_{\mu\part n}
{\TH_\mu[X]MB_\mu \Pi_\mu\over w_\mu}
%\eqno 1.35
\label{eq1p35}
\end{equation}
This is perhaps the simplest way to prove \eqref{eq1p32} f). 
This discovery is precisely what led to  the introduction of $\nabla$ in the first place.
\end{remark}
\end{section}
\vfill
\newpage

\begin{section}{Proof of the basic recursion}

To establish  Theorem \ref{th:IIp1} we need some preliminary observations.
To begin we have the following reduction.
\sas

%Theorem 2.1}
\begin{theorem}\label{th:2p1}
For all  $p=(p_1,p_2,\ldots ,p_k)\models n$  and $j\ge 0$ we have

\begin{align}
\LL \Delta_{h_j}\BC_{p_1}\BC_{p_2}\cdots \BC_{p_k} 1\scs  e_n \RR 
&\ses
t^{p_1-1}q^{k-1}
\LL \Delta_{h_{j-1}} \BC_{p_2}\cdots \BC_{p_k}\BB_{p_1} 1\scs  
 e_{n} \RR \nonumber\\
&\bigsp\bigsp
 \sps \chi(p_1=1)\LL \Delta_{h_{j}}\BC_{p_2}\cdots \BC_{p_k} 1\scs   e_{n-1} \RR 
\label{eq2p1}
%\eqno 2.1
\end{align}
if and only if , with  $\BC_a^*$ and $\BB_a^*$  the $*$-scalar product duals
of $\BC_a $ and $\BB_a $,  we have
\begin{equation}
\BC^*_a\Delta_{h_j}h_n[\tttt{X\over M}]
\ses
t^{a-1}\BB^*_{a} 
 \Delta_{h_{j-1}}h_{n}[\tttt{X\over M}]
 \sps \chi(a=1)\Delta_{h_j}h_{n-1}[\tttt{X\over M}]
%\eqno 2.2 
\label{eq2p2}
\end{equation}
for all  $j\ge 0$ and $1\le a\le n$. 
\end{theorem}

\begin{proof}
It is shown in \cite{HMZ} that the operators $\BC_a$ and $\BB_b$ satisfy the commutativity  relations 
$$
q\BC_a\BB_b\ses \BB_b\BC_a
\ess\ess (\hbox{for all $a,b\ge 1$})
$$
Using these identities, \eqref{eq2p1} becomes
\begin{align*}
\LL \Delta_{h_j} \BC_{p_1}\BC_{p_2}\cdots \BC_{p_k} 1\scs  e_n \RR 
&\ses
t^{p_1-1 }
\LL \Delta_{h_{j-1}}\BB_{p_1}  \BC_{p_2}\cdots \BC_{p_k}1\scs  
 e_{n} \RR 
 \\
 &\bigsp\bigsp
 \sps \chi(p_1=1)\LL \Delta_{h_{j}}\BC_{p_2}\cdots \BC_{p_k} 1\scs   e_{n-1} \RR. 
\end{align*}
Passing to $*$-scalar products, and using the identity in \eqref{eq1p29},
we can next rewrite \eqref{eq2p1} in the form
\begin{equation}
\LL \Delta_{h_j}\BC_{p_1}F[X]\scs  h_n^* \RR_* 
\ses
t^{p_1-1} \LL \Delta_{h_{j-1}}\BB_{p_1} F[X]\scs  
 h_n^* \RR_* 
 \sps \chi(p_1=1)\LL \Delta_{h_{j}}F[X]\scs   h_{n-1}^* 
 \RR_* 
%\eqno 2.3
\label{eq2p3}
\end{equation}
and the validity of this identity for every symmetric function  $F[X]$ 
that is homogeneous of degree $n-p_1$  is equivalent to \eqref{eq2p1} since when
$p_2,\ldots p_k$ are the parts of a partition the polynomials
$
\BC_{p_2}\cdots \BC_{p_k} 1
$
are essentially elements of the Hall-Littlewood basis. Now, since 
all the operators $\DD_F$ are self adjoint with respect to the $*$-scalar product, 
\eqref{eq2p3} in turn can be rewritten in the form
$$
\LL F[X]\scs  \BC_{p_1}^* \Delta_{h_j}h_n^* \RR_* 
\ses
t^{p_1-1} 
\LL  F[X]\scs  
\BB_{p_1}^* \Delta_{h_{j-1}} h_n^* \RR_* 
 \sps \chi(p_1=1)\LL F[X]\scs  \Delta_{h_{j}} h_{n-1}^* 
 \RR_* 
$$
and this identity (for all $p_1\ge 1$) is equivalent to \eqref{eq2p2} due to the arbitrariness of $F[X]$. This completes  our proof.
\end{proof}

Our next goal is to prove \eqref{eq2p2}.
To begin we have the following auxiliary identity.

%Proposition 2.1}
\begin{prop} \label{prop:2p1}
\begin{equation}
\DD _{h_j}h_n\big[\tttt{X\over M} \big]
\ses \sum_{s=0}^jh_{j-s}[\tttt{1\over M} ](-1)^{s}
 \sum_{\nu\part s}{T_\nu^2\over w_\nu}
h_n\big[X(\tttt {1\over M}-B_\nu)\big].
%\eqno 2.4
\label{eq2p4}
\end{equation}
\end{prop}

\begin{proof}
Using \eqref{eq1p32} c) and the definition of the operator $\DD_{h_j}$ we get
\begin{align*}
\Delta_{h_{j}}h_{n}[\tttt{X\over M}]
&\ses \sum_{\mu\part n}{T_\mu\TH_\mu[X]\over w_\mu}h_{j}[B_\mu]
\\
&\ses \sum_{\mu\part n}{T_\mu\TH_\mu\over w_\mu}
h_{j}\big[\tttt{MB_\mu-1\over M}+\tttt{1\over M}\big]
\\
&\ses \sum_{\mu\part n}{T_\mu\TH_\mu[X]\over w_\mu}
\sum_{s=0}^jh_{j-s}\big[\tttt{1\over M}\big]h_{s}\big[\tttt{MB_\mu-1\over M}\big]
\\
 &\ses \sum_{s=0}^jh_{j-s}\big[\tttt{1\over M}\big]
\sum_{\mu\part n}{T_\mu\TH_\mu[X]\over w_\mu}
\, h_{s}\big[\tttt{MB_\mu-1\over M}\big]
\\
 &\ses \sum_{s=0}^jh_{j-s}\big[\tttt{1\over M}\big]
\sum_{\mu\part n}{T_\mu\TH_\mu[X]\over w_\mu}
\,\sum_{\nu\part s}{T_\nu\TH_\nu[MB_\mu-1]\over w_\nu}
\\
 &\ses \sum_{s=0}^jh_{j-s}\big[\tttt{1\over M}\big]
\sum_{\nu\part s}{T_\nu^2\over w_\nu}
\sum_{\mu\part n}{T_\mu\TH_\mu[X]\over w_\mu}
\,{\TH_\nu[MB_\mu-1]\over T_\nu}
\\ (\hbox{by the reciprocity in \eqref{eq1p14}})
 &\ses \sum_{s=0}^jh_{j-s}\big[\tttt{1\over M}\big]
\sum_{\nu\part s}{T_\nu^2\over w_\nu}(-1)^{n-s}
\sum_{\mu\part n}{ \TH_\mu[X]\TH_\mu[MB_\nu-1]\over w_\mu}
\\
 (\hbox{by \eqref{eq1p31})})&\ses 
\sum_{s=0}^jh_{j-s}\big[\tttt{1\over M}\big]
\sum_{\nu\part s}{T_\nu^2\over w_\nu}(-1)^{n-s}
e_n\big[X(B_\nu - \tttt{1\over M} )\big]
\end{align*}
This gives \eqref{eq2p4} since $e_n\big[ X(B_\nu-\tttt {1\over M})\big]=(-1)^n
h_n\big[X(\tttt {1\over M}-B_\nu)\big]$.
\end{proof}

It is good to keep in mind that in particular we also have
\begin{equation}
\Delta_{h_{j}}h_{n-1}[\tttt{X\over M}]\ses 
\sum_{s=0}^j(-1)^s h_{j-s}\big[\tttt{1\over M}\big]
\sum_{\nu\part s}{T_\nu^2\over w_\nu} 
h_{n-1}\big[X(\tttt {1\over M}-B_\nu)\big]
%\eqno 2.5
\label{eq2p5}
\end{equation}
\begin{equation}
\Delta_{h_{j-1}}h_{n}[\tttt{X\over M}]\ses 
\sum_{s=0}^{j-1}(-1)^sh_{j-1-s}\big[\tttt{1\over M}\big]
\sum_{\nu\part s}{T_\nu^2\over w_\nu} 
h_n\big[X(\tttt {1\over M}-B_\nu)\big]
%\eqno 2.6
\label{eq2p6}
\end{equation}

Next we have

%Proposition 2.2}
\begin{prop}\label{prop:2p2}
\begin{align}
 \BC^*_a\DD _{h_j}&h_n\big[\tttt{X\over M} \big] 
 \sms \chi(a=1)\DD _{h_j}h_{n-1}\big[\tttt{X\over M} \big]\ses \label{eq2p7}\\
&\sps 
t^{a-1} \sum_{s=0}^{j-1}h_{j-1-s}[\tttt{1\over M} ](-1)^{s}
\sum_{\tau\part s}{T_\tau^2\over w_\tau}
\sum_{v=a}^{n}
h_{v-a}\big[ \tttt{-X\over 1-t }\big]
h_{n-v}\big[ X (\tttt {1\over M}-B_\tau) \big]
h_{v}\big[1-MB_\tau  \big]\nonumber
%\eqno 2.7
\end{align}
\end{prop}
\begin{proof}
The identity in \eqref{eq2p4}  gives
\begin{equation}
\BC^*_a\DD _{h_j}h_n\big[\tttt{X\over M} \big]
\ses \sum_{s=0}^j h_{j-s}[\tttt{1\over M} ](-1)^{s}
 \sum_{\nu\part s}{T_\nu^2\over w_\nu}
\BC^*_ah_n\big[X(\tttt {1\over M}-B_\nu)\big].
%\eqno 2.8
\label{eq2p8}
\end{equation}
Now it was shown in \cite{GXZ} that for all $P[X]\in \La$ we have
$$
\BC_a^* P[X]\ses(\tttt{ -1\over q})^{a-1} 
P\big[X-\tttt{\eee M\over z}  \big]\OM\big[\tttt{-\eee zX\over q(1-t)} \big]
\Big|_{z^{-a}}
$$
This gives
\begin{align*}
(-q)^{a-1} \BC^*_a
h_n\big[X(\tttt {1\over M}-B_\nu)\big]
&\ses
h_n\big[(X-\tttt{\eee M\over z}) (\tttt {1\over M}-B_\nu) \big]
\OM\big[\tttt{-\eee zX\over q(1-t)} \big]
\Big|_{z^{-a}}
\\
&\ses
\sum_{r=0}^n
h_{n-r}\big[ X (\tttt {1\over M}-B_\nu) \big]
h_r\big[\tttt{-\eee M\over z} (\tttt {1\over M}-B_\nu) \big]
\OM\big[\tttt{-\eee zX\over q(1-t)} \big]
\Big|_{z^{-a}}
\\
&\ses
\sum_{r=0}^n
h_{n-r}\big[ X (\tttt {1\over M}-B_\nu) \big]
(-1)^rh_r\big[   M B_\nu-1 \big]
\OM\big[\tttt{-\eee zX\over q(1-t)} \big]
\Big|_{z^{r-a}}
\\
&\ses
\sum_{r=0}^n
h_{n-r}\big[ X (\tttt {1\over M}-B_\nu) \big]
(-1)^rh_r\big[   M B_\nu-1 \big]
h_{r-a}\big[\tttt{-\eee X\over q(1-t)} \big]
\\
&\ses
\sum_{r=a}^n
h_{n-r}\big[ X (\tttt {1\over M}-B_\nu) \big]
\tttt{(-1)^a\over q^{r-a}}h_r\big[   M B_\nu-1 \big]
h_{r-a}\big[\tttt{ - X\over 1-t } \big]
\end{align*}
or better
$$
\BC^*_a
h_n\big[X(\tttt {1\over M}-B_\nu)\big]
\ses \sms \sum_{r=a}^n
h_{n-r}\big[ X (\tttt {1\over M}-B_\nu) \big]
\tttt{1\over q^{r-1}}h_r\big[   M B_\nu-1 \big]
h_{r-a}\big[\tttt{ - X\over 1-t } \big].
$$
and the last sum in \eqref{eq2p8} becomes
\begin{align}
  \sum_{\nu\part s}{T_\nu^2\over w_\nu}\BC^*_a
h_n\big[X(\tttt {1\over M}-B_\nu)\big]
&=
\sms \sum_{\nu\part s}{T_\nu^2\over w_\nu}\sum_{r=a}^n
h_{n-r}\big[ X (\tttt {1\over M}-B_\nu) \big]
\tttt{1\over q^{r-1}}h_r\big[   M B_\nu-1 \big]
h_{r-a}\big[\tttt{  -X\over 1-t } \big]\nonumber
\\
&=
\sms 
\sum_{r=a}^n\tttt{1\over q^{r-1}}h_{r-a}\big[\tttt{  -X\over 1-t } \big]
\sum_{\nu\part s}{T_\nu^2\over w_\nu}
h_{n-r}\big[ X (\tttt {1\over M}-B_\nu) \big]
h_r\big[   M B_\nu-1 \big]
\label{eq2p9}
%\eqno 2.9
\end{align}
Using the summation identity in \eqref{eq1p22} in the form 
$$
h_{r}\left[ MB_\nu -1\right]\ses 
(tq)^{r-1} \sum_{\tau\RA \nu }Mc_{\nu\tau}(\tttt{T_\nu\over T_\tau})^{r-1} 
\sms \chi(r=1) 
$$
we get
\begin{align*}
  \sum_{\nu\part s}{T_\nu^2\over w_\nu}\BC^*_a
h_n\big[X(\tttt {1\over M}-B_\nu)\big]
&=
\sms 
\sum_{r=a}^n\tttt{1\over q^{r-1}}h_{r-a}\big[\tttt{  -X\over 1-t } \big]
\sum_{\nu\part s}{T_\nu^2\over w_\nu}
h_{n-r}\big[ X (\tttt {1\over M}-B_\nu) \big]
(tq)^{r-1} \sum_{\tau\RA \nu }Mc_{\nu\tau}(\tttt{T_\nu\over T_\tau})^{r-1} 
\\
&\bigsp\bigsp
\sps 
\sum_{r=a}^n\tttt{1\over q^{r-1}}h_{r-a}\big[\tttt{  -X\over 1-t } \big]
\sum_{\nu\part s}{T_\nu^2\over w_\nu}
h_{n-r}\big[ X (\tttt {1\over M}-B_\nu) \big]
\chi(r=1)
\\
&=
\sms 
\sum_{r=a}^n
\tttt{t^{r-1}}h_{r-a}\big[\tttt{  -X\over 1-t } \big]
\sum_{\nu\part s}{T_\nu^2\over w_\nu}
h_{n-r}\big[ X (\tttt {1\over M}-B_\nu) \big]
 \sum_{\tau\RA \nu }Mc_{\nu\tau}(\tttt{T_\nu\over T_\tau})^{r-1} 
\\
&\bigsp\bigsp\bigsp\bigsp
\sps 
\chi(a=1)
\sum_{\nu\part s}{T_\nu^2\over w_\nu}
h_{n-1}\big[ X (\tttt {1\over M}-B_\nu) \big]
\\
&=
\sms t^{a-1}
\sum_{r=a}^n
h_{r-a}\big[\tttt{  -tX\over 1-t } \big]
\sum_{\tau\part s-1}{T_\tau^2\over w_\tau}
\sum_{\nu\LA \tau }
h_{n-r}\big[ X (\tttt {1\over M}-B_\nu) \big]
 M\tttt{w_\tau\over w_\nu}c_{\nu\tau}(\tttt{T_\nu\over T_\tau})^{r+1} 
\\
&\bigsp\bigsp\bigsp\bigsp
\sps 
\chi(a=1)
\sum_{\nu\part s}{T_\nu^2\over w_\nu}
h_{n-1}\big[ X (\tttt {1\over M}-B_\nu) \big]
\end{align*}
\vskip -.1in

\noindent
Using \eqref{eq1p24} and the fact that there are no partitions of size $-1$,
\eqref{eq2p8} becomes
\begin{align*}
\BC^*_a\DD _{h_j}h_n\big[\tttt{X\over M} \big]
&\ses
\sms t^{a-1} \sum_{s=1}^jh_{j-s}[\tttt{1\over M} ](-1)^{s}
\sum_{r=a}^n
h_{r-a}\big[\tttt{  -tX\over 1-t } \big]
\sum_{\tau\part s-1}{T_\tau^2\over w_\tau}
\sum_{\nu\LA \tau }
h_{n-r}\big[ X (\tttt {1\over M}-B_\nu) \big]
d_{\nu\tau}(\tttt{T_\nu\over T_\tau})^{r+1} 
\\
&\bigsp\bigsp\bigsp\bigsp
\sps 
\chi(a=1)\sum_{s=0}^jh_{j-s}[\tttt{1\over M} ](-1)^{s}
\sum_{\nu\part s}{T_\nu^2\over w_\nu}
h_{n-1}\big[ X (\tttt {1\over M}-B_\nu) \big]
\end{align*}

\noindent and \eqref{eq2p5} gives 
\begin{align}
 \BC^*_a\DD _{h_j}h_n\big[\tttt{X\over M} \big] 
&= 
- t^{a-1} \sum_{s=1}^jh_{j-s}[\tttt{1\over M} ](-1)^{s}
\sum_{r=a}^n
h_{r-a}\big[\tttt{  -tX\over 1-t } \big]
\sum_{\tau\part s-1}{T_\tau^2\over w_\tau}
\sum_{\nu\LA \tau }
h_{n-r}\big[ X (\tttt {1\over M}-B_\nu) \big]
d_{\nu\tau}(\tttt{T_\nu\over T_\tau})^{r+1} \nonumber
\\
&\bigsp\bigsp \bigsp\bigsp\bigsp\bigsp\bigsp\sps
\chi(a=1)\DD _{h_j}h_{n-1}\big[\tttt{X\over M} \big]
\label{eq2p10}
%\eqno 2.10
\end{align}
Since $B_\nu=B_\tau+{T_\nu\over T_\tau}$ 
 we derive that
\begin{align*}
\sum_{\nu\LA \tau }
h_{n-r}\big[ X (\tttt {1\over M}-B_\nu) \big]
d_{\nu\tau}(\tttt{T_\nu\over T_\tau})^{r+1} 
&\ses
\sum_{u=0}^{n-r}
h_{n-r-u}\big[ X (\tttt {1\over M}-B_\tau) \big]h_{u}\big[ -X \big]
\sum_{\nu\LA \tau }d_{\nu\tau}(\tttt{T_\nu\over T_\tau})^{r+u+1}
\\
&\ses
\sum_{v=r}^{n}
h_{n-v}\big[ X (\tttt {1\over M}-B_\tau) \big]h_{v-r}\big[ -X \big]
\sum_{\nu\LA \tau }d_{\nu\tau}(\tttt{T_\nu\over T_\tau})^{v+1}
\end{align*}
Using the summation formula in \eqref{eq1p23} in the form
\def \mur {\leftarrow}
$$
\sum_{\mu\mur\nu}d_{\mu\nu}(q,t)\, (T_\mu/T_\nu)^k\ses 
\begin{cases}
 h_{k-1}\big[1-MB_\nu  \big] &\hbox{ if }k\geq 1\\
1 & if k=0\ess .
\end{cases}$$
together with the fact that  $v\ge r$ and  in \eqref{eq2p10} we have $r\ge a\ge 1$ we 
obtain
$$
\sum_{\nu\LA \tau }d_{\nu\tau}(\tttt{T_\nu\over T_\tau})^{v+1}
\ses h_{v}\big[1-MB_\tau  \big]
$$
and \eqref{eq2p10} becomes
\begin{align*}
\BC^*_a\DD _{h_j}&h_n\big[\tttt{X\over M} \big] 
\sms \chi(a=1)\DD _{h_j}h_{n-1}\big[\tttt{X\over M} \big] \\
&= 
- t^{a-1} \sum_{s=1}^jh_{j-s}[\tttt{1\over M} ](-1)^{s}
\sum_{r=a}^n
h_{r-a}\big[\tttt{  -tX\over 1-t } \big]
\sum_{\tau\part s-1}\Big({T_\tau^2\over w_\tau}\\
&\hskip .5in \sum_{v=r}^{n}
h_{n-v}\big[ X (\tttt {1\over M}-B_\tau) \big]h_{v-r}\big[ -X \big]
h_{v}\big[1-MB_\tau  \big]\Big)\\
&= 
- t^{a-1} \sum_{s=1}^jh_{j-s}[\tttt{1\over M} ](-1)^{s}
\sum_{\tau\part s-1}{T_\tau^2\over w_\tau}
\sum_{v=a}^{n}\sum_{r=a}^v
\Big(h_{v-r}\big[ \tttt{-(1-t)X\over 1-t }\big]\\
&\hskip .5in h_{r-a}\big[\tttt{  -tX\over 1-t } \big]
h_{n-v}\big[ X (\tttt {1\over M}-B_\tau) \big]
h_{v}\big[1-MB_\tau  \big]\Big)\\
&= 
- t^{a-1} \sum_{s=1}^jh_{j-s}[\tttt{1\over M} ](-1)^{s}
\sum_{\tau\part s-1}{T_\tau^2\over w_\tau}
\sum_{v=a}^{n}
h_{v-a}\big[ \tttt{-X\over 1-t }\big]
h_{n-v}\big[ X (\tttt {1\over M}-B_\tau) \big]
h_{v}\big[1-MB_\tau  \big]
\end{align*}
 This best rewritten in the form
\begin{align*}
 \BC^*_a\DD _{h_j}&h_n\big[\tttt{X\over M} \big] 
 \sms \chi(a=1)\DD _{h_j}h_{n-1}\big[\tttt{X\over M} \big]\ses \\
&= 
t^{a-1} \sum_{s=0}^{j-1}h_{j-1-s}[\tttt{1\over M} ](-1)^{s}
\sum_{\tau\part s}{T_\tau^2\over w_\tau}
\sum_{v=a}^{n}
h_{v-a}\big[ \tttt{-X\over 1-t }\big]
h_{n-v}\big[ X (\tttt {1\over M}-B_\tau) \big]
h_{v}\big[1-MB_\tau  \big]
\end{align*}
proving \eqref{eq2p7} and completing our proof of Proposition \ref{prop:2p2}.
\end{proof}

Let us now work on $\BB^*_{a} 
 \Delta_{h_{j-1}}h_{n}[\tttt{X\over M}]
$. 
Here we use the identity in \eqref{eq2p6}, that is  
$$
\DD _{h_{j-1}}h_n\big[\tttt{X\over M} \big]
\ses \sum_{s=0}^jh_{j-1-s}[\tttt{1\over M} ](-1)^{s}
 \sum_{\nu\part s}{T_\nu^2\over w_\nu}
h_n\big[X(\tttt {1\over M}-B_\nu)\big].
$$
and the  identity
$$
\BB_a^* P[X]\ses P\big[X+\tttt{M\over z}  
\big]\OM\big[\tttt{-zX\over 1-t} \big]
\Big|_{z^{-a}}
$$
 (established in \cite{GXZ})  that gives  the action of the operators  $\BB_a^*$
to obtain
\begin{align*}
\BB_a^* \Delta_{h_{j-1}}h_{n}[\tttt{X\over M}]
&\ses \sum_{s=0}^{j-1}(-1)^se_{j-1-s}\big[\tttt{1\over M}\big]
\sum_{\nu\part s}{T_\nu^2\over w_\nu} 
\BB_a^* h_n\big[X( \tttt{1\over M}-B_\nu  )\big]
\\
&\ses \sum_{s=0}^{j-1}(-1)^sh_{j-1-s}\big[\tttt{1\over M}\big]
\sum_{\nu\part s}{T_\nu^2\over w_\nu} 
h_n\big[(X+\tttt{M\over z})( \tttt{1\over M}-B_\nu  )
\big]\OM\big[\tttt{-zX\over 1-t} \big]
\Big|_{z^{-a}}
\\
&\ses \sum_{s=0}^{j-1}(-1)^sh_{j-1-s}\big[\tttt{1\over M}\big]
\sum_{\nu\part s}{T_\nu^2\over w_\nu} 
\sum_{r=0}^n
h_{n-r}\big[X( \tttt{1\over M}-B_\nu  )\big]
h_{r}\big[MB_\nu - 1 \big]
\OM\big[\tttt{-zX\over 1-t} \big]
\Big|_{z^{-a+r}}
\\
&\ses \sum_{s=0}^{j-1}(-1)^sh_{j-1-s}\big[\tttt{1\over M}\big]
\sum_{\nu\part s}{T_\nu^2\over w_\nu} 
\sum_{r=a}^n
h_{n-r}\big[X( \tttt{1\over M}-B_\nu  )\big]
h_{r}\big[1-MB_\nu \big]
h_{r-a}\big[\tttt{-X\over 1-t} \big]
\end{align*}
and this may be rewritten as
$$
\BB_a^* \Delta_{h_{j-1}}h_{n}[\tttt{X\over M}]
\ses \sum_{s=0}^{j-1}(-1)^sh_{j-1-s}\big[\tttt{1\over M}\big]
\sum_{\nu\part s}{T_\nu^2\over w_\nu} 
\sum_{r=a}^n h_{r-a}\big[\tttt{-X\over 1-t} \big]
h_{n-r}\big[X( \tttt{1\over M}-B_\nu  )\big]
h_{r}\big[1-MB_\nu  \big]
$$
Comparing with the right hand side of \eqref{eq2p7} we see that we have established the identity
$$
t^{a-1} \BB_a^* \Delta_{h_{j-1}}h_{n}[\tttt{X\over M}]
\ses \BC^*_a\DD _{h_j}h_n\big[\tttt{X\over M} \big] 
 \sms \chi(a=1)\DD _{h_j}h_{n-1}\big[\tttt{X\over M} \big]
 $$
This completes our proof of \eqref{eq2p2} an consequently also the proof of Theorem \ref{th:IIp1}.

\end{section}
\vfill
\newpage

\begin{section}{The construction of the new dinv }

Let $\CAP(J,n)$ denote the collection of Parking functions on the $J+n\times J+n$ 
lattice square whose diagonal word is a shuffle of the two  words
$\EE_J=12\ldots J$ and $\EE_{J,n}=J+1\ldots J+n$ with   car $J+n$ in the (1,1) lattice square. In symbols
\begin{equation}
\CAP(J,n)\ses \big\{PF\in \CAP_{J+n}\,:\, \sig(PF)\in  \EE_J\sch \EE_{n-J}
\ssp \&\ssp J+n\in (1,1)  \big\}
%\eqno 3.1
\label{eq3p1}
\end{equation}

Before  we can proceed with our construction of the new dinv, we need
some preliminary observations about this family of parking functions.
To begin we should note that the condition that the diagonal word
be a shuffle of $12\cdots J$ with $J+1\cdots J+n$, together with the  
column increasing property of parking functions, forces the
columns of the Dyck path
\footnote {$\hskip -.07in\ ^{(*)}$}
{Where by the {\it length} of column $i$ of a Dyck path $D$ 
we refer to the number of NORTH steps of $D$ of abscissa $i$}
 supporting a $PF\in \CAP(J,n)$ to be of length $2$ at most.
 The reason for this is simple: as we read the cars of $PF$ to obtain $\sig(PF)$ from right to left by diagonals starting from the highest and ending with the lowest
the big cars ($J+1,\ldots ,J+n$) as well as the small cars 
($1,2,\ldots ,J$) will be  increasing.
Thus we will never see a big car on top of a big car nor a small car on top of a small car. So the only possibility is a big car on top of a small car,
i. e. columns of length $2$ at most as we asserted.

This yields an algorithm for constructing all the elements of the family 
$\CAP(J,n)$. Let us  denote by ``$red(PF)$'', and call it the ``{\it reduced 
tableau }'' of $PF$, the configuration obtained by replacing 
in a $PF\in \CAP(J,n)$ all big cars by a $2$ an all small cars by a $1$. 
We can simply obtain all the reduced tableaux of elements of $\CAP(J,n)$
by constructing first the family $\CD_{J,n}$ of  Dyck paths of length $n+J$ 
with no more than $J$ columns of length $2$ and all remaining columns of 
length $1$. Then for   each Dyck path $D\in \CD_{J,n}$
fill the cells adjacent  to the NORTH steps of each  column of length $2$ 
 by a $1$ under a $2$, then fill the columns of length $1$ by a $1$ or a $2$
 for a total of $J$ ones and $n$  twos. 
 
  Clearly each $PF\in \CAP(J,n)$
 can be uniquely reconstructed from its reduced tableau by replacing
 all the ones by $1,2,\ldots ,J$ and all the twos by $J+1\ldots J+n$ 
 by diagonals from right to left starting from the highest and ending with the lowest. It will  also be clear that we need only work with reduced tableaux 
 to construct our new dinv.  However,  being able to refer to the original cars
 will turn out to be more  convenient in some of our proofs. For this reason
 we will work with a $PF$ or its $red(PF)$  intercheangeably
 depending on the context.
 
 This given,  we have the following basic fact
\sas

%Proposition 3.1}
\begin{prop}\label{prop:3p1}
For any 
$$
PF=\begin{bmatrix}
v_1 & v_2 &  \cdots  &  v_n  \\
u_1 &  u_2 & \cdots  &  u_n  \\
\end{bmatrix}
\in \CAP(J,n)
$$  
if we set
$
\{i_1 <i_2<\cdots <i_k \}
\ses \big\{i\in [1 ,J+n]  \, :\, v_i >J  \big\}
$
then the  vector 
$$
U_B(PF) \ses (u_{i_1},u_{i_2},\ldots ,u_{i_k})
$$
gives the area sequence of a Dyck path, which  here and after 
will be referred to as the Dyck path ``supporting'' the big cars of PF.
\end{prop}

\begin{proof}
Since car $J+n$ is in the (1,1) lattice square it follows that $u_{i_1}=0$.
Thus we need only show that
$$
u_{i_s}\le u_{i_{s-1}}+1\ess\ess\ess \hbox{ for all $2\le s\le k$} 
$$
By definition $v_{i_{s-1}}$ and  $v_{i_{s}}$ are two successive big cars in $PF$
that means that for $i_{s-1}<j<i_s$ the car $v_j$ is small and thus, except perhaps 
for $j=i_s-1$, the car $v_j$ must be in a column of length $1$. In particular we see
that we must have $u_j\le u_{i_{s-1}}$ for,  the first  violation of this inequality 
would put a small car above a big car (for $j=i_{s-1}+1$) or a small car above a 
small car for $j>i_{s-1}+1$. This gives $u_{i_s} \le u_{i_s-1}+1$
as desired  with equality only if car $v_{i_s}$ is at the top of a column of
length $2$ and all the small cars in between $v_{i_{s-1}}$ and $v_{i_s}$
are in the same diagonal as $v_{i_{s-1}}$. 
\end{proof}
\sas

In view of this result we are now going to focus on the subfamilies   $\CAP_{J}(p)$
of $\CAP_{J,n}$ consisting of its elements whose big cars have a supporting Dyck path which hits the diagonal according to a given composition $p=(p_1,p_2,\ldots ,p_k)\models n$. Our goal is to construct a statistic ``{\it ndinv }'' which yields the equality.
\begin{equation}
\LL \DD_{h_J}\BC_{p_1}\BC_{p_1}\cdots \BC_{p_k} 1\scs e_n\RR \ses
\sum_{PF\in \CAP_{J}(p)} t^{area(PF)} q^{ndinv(PF)}
%\eqno 3.2
\label{eq3p2}
\end{equation}
But before we do this it may be good to experiment a little   by constructing some of these families.

By reversing the argument we used to prove  Proposition \ref{prop:3p1}, we can start by constructing all the Dyck paths with the given diagonal composition then ad
all the cars as required by the definition of a family. This is best illustrated by
examples. Say we start with with $p=(3,2)$. In this case there are only two 
possible Dyck paths as given below on the left .
$$
\hbox{ \includegraphics[width=2in]{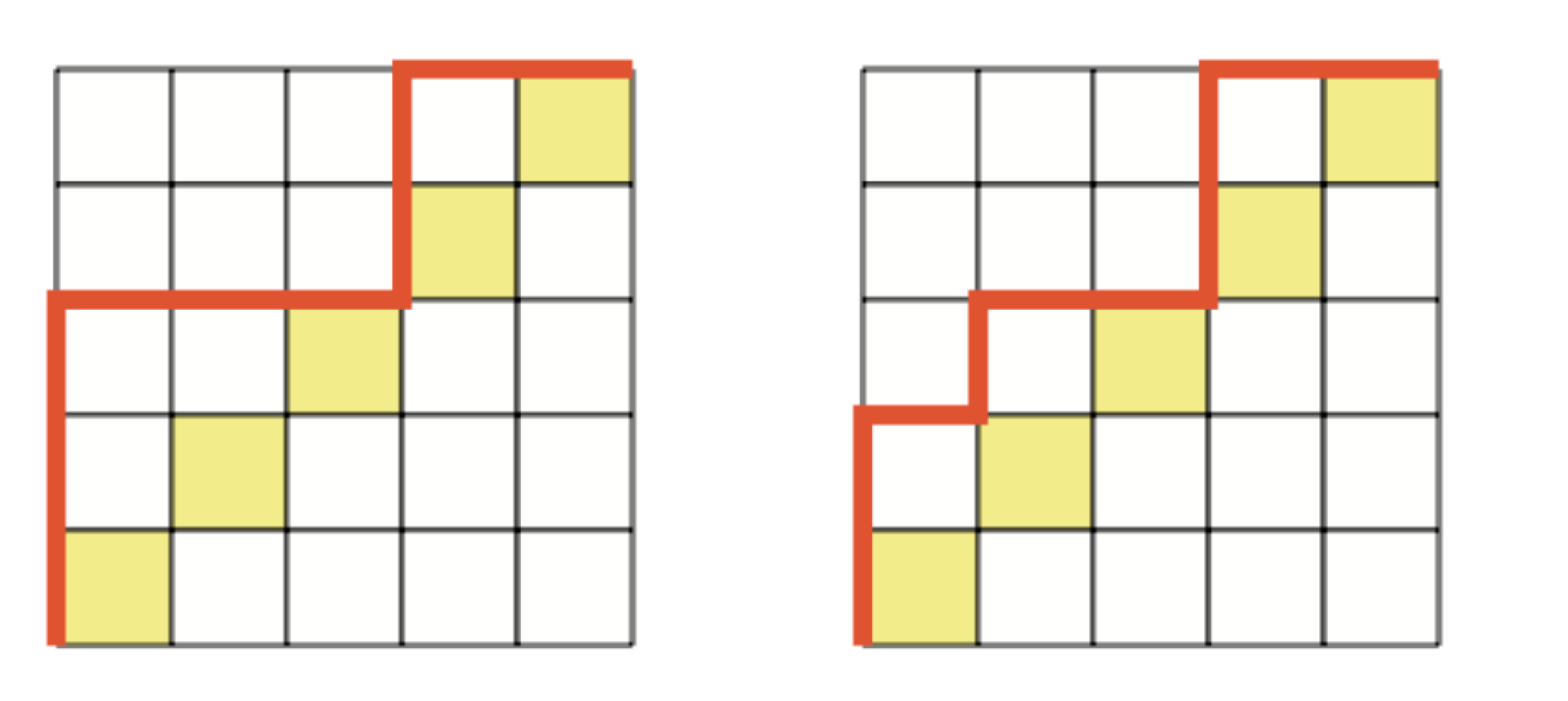}}  
\bigsp
\hbox{ \includegraphics[width=2in]{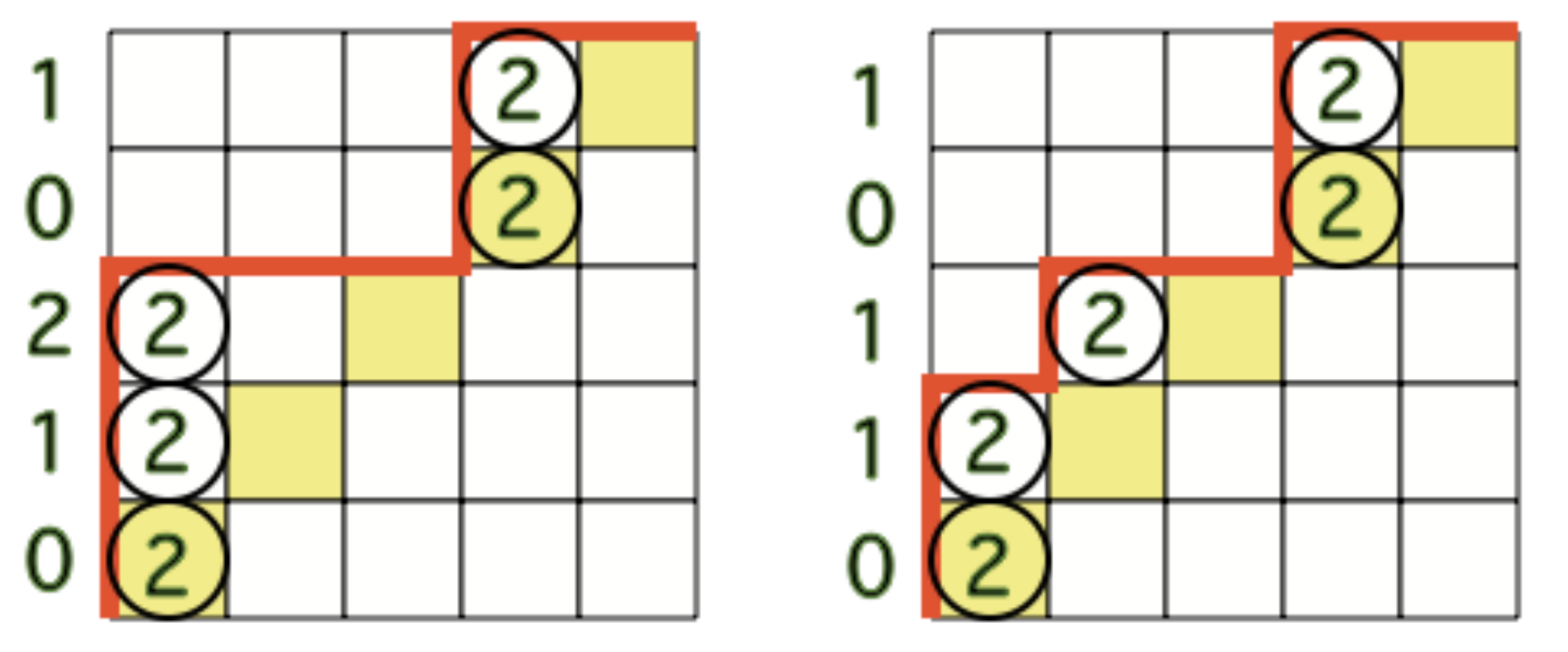}}  
$$
On the right we added the $2's$ and their corresponding diagonal numbers.
Now the least number of $1's$ we need to add to get a legal reduced diagram 
is $3$ for the first and 2 for the second as shown below
\begin{equation}
\vcenter{\hbox{ \includegraphics[width=1.5in]{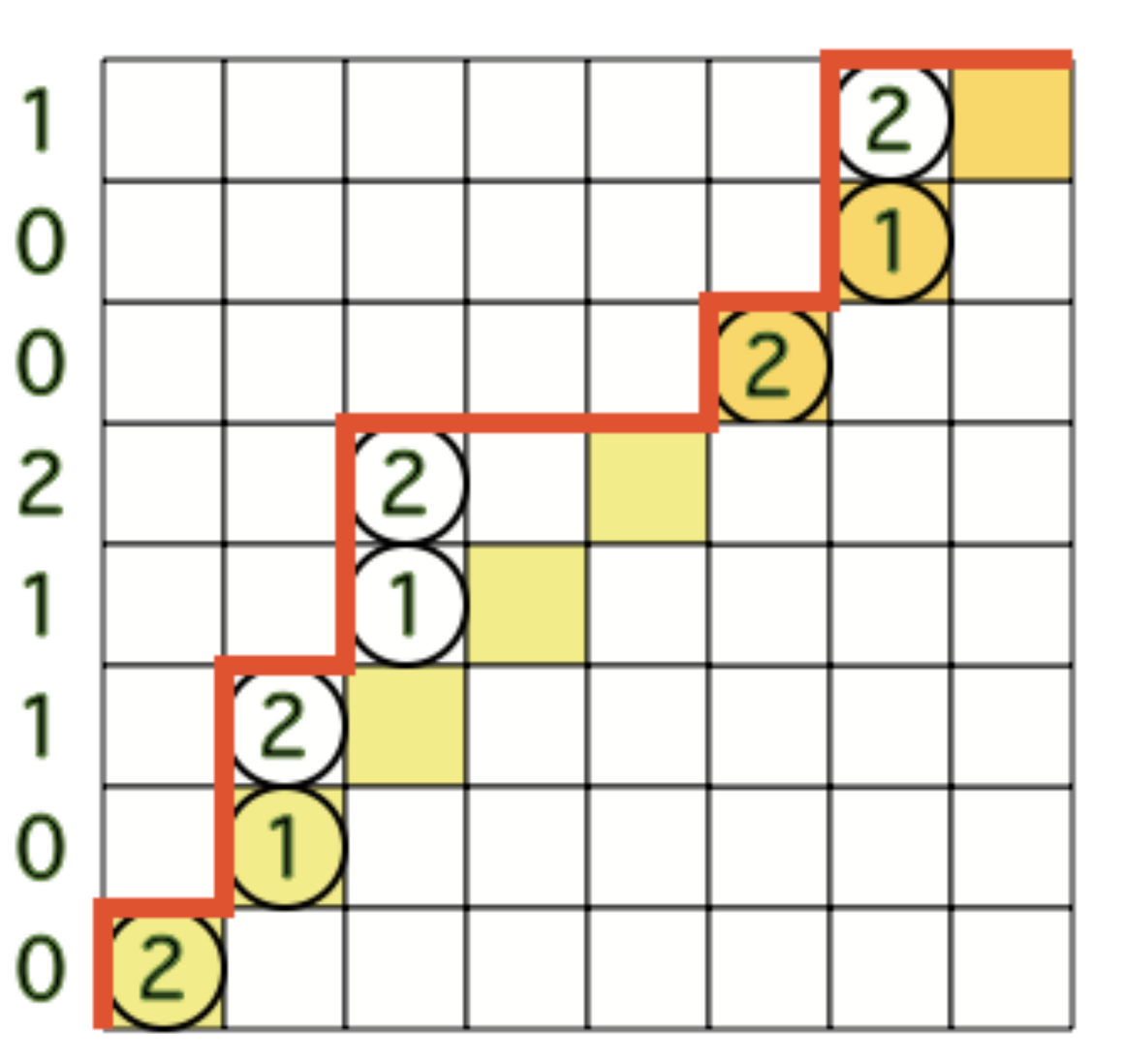}}}  
\bigsp
\vcenter{\hbox{ \includegraphics[width=1.3in]{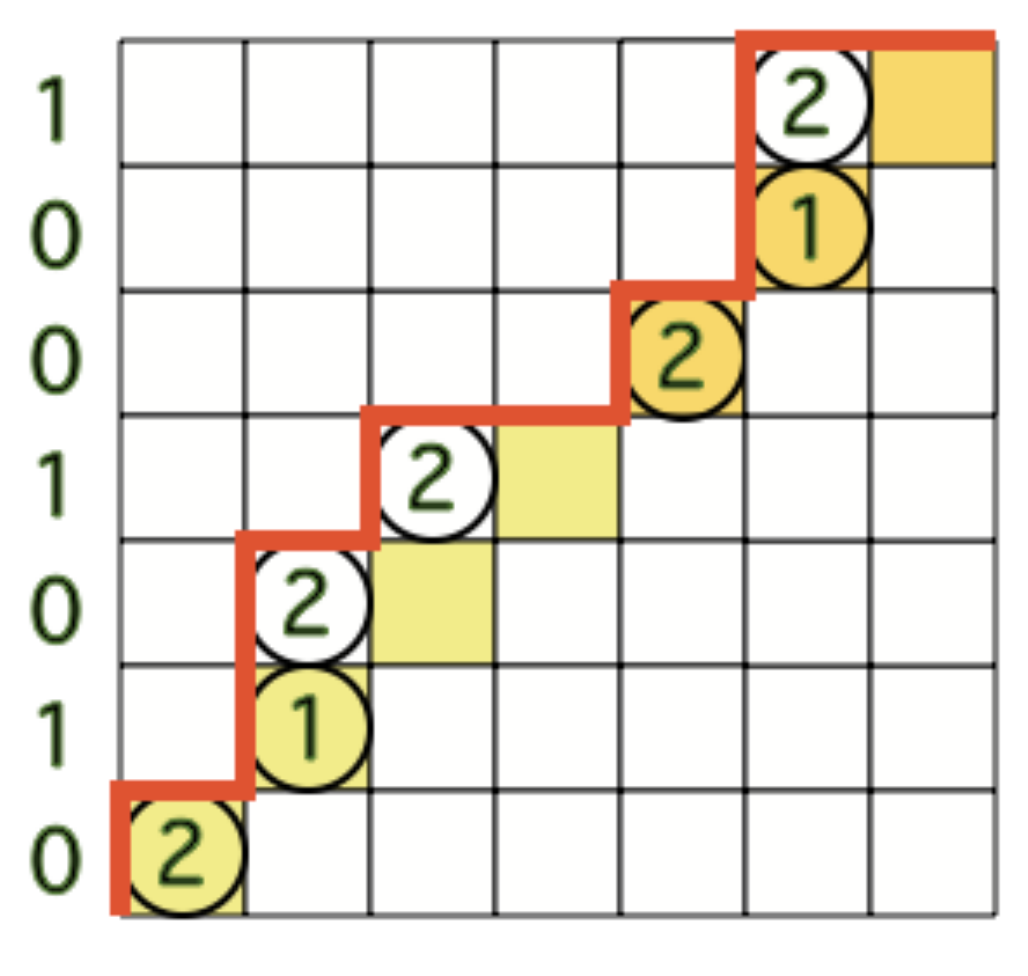}}}  
%\eqno 3.4
\label{eq3p4}
\end{equation}
Now a MAPLE computation yields the polynomials
\begin{equation}
\LL \DD_{h_2}\BC_3\BC_2  1\scs e_5\RR\ses t^3 q^4
%\eqno 3.5
\label{eq3p5}
\end{equation}
and
\begin{equation}
\LL \DD_{h_3}\BC_3\BC_2  1\scs e_5\RR\ses
t^3(q^4+q^5+q^6+q^7+q^8)+t^4(q^4+q^5+q^6)+t^5 q^4~.
%\eqno 3.6
\label{eq3p6}
\end{equation}

\vfill
To compute the classical weight $t^{area}q^{dinv}$ it is better to 
have a look at the non-reduced versions of the two tableau above.
Namely
\begin{equation}
\vcenter{\hbox{ \includegraphics[width=1.5in]{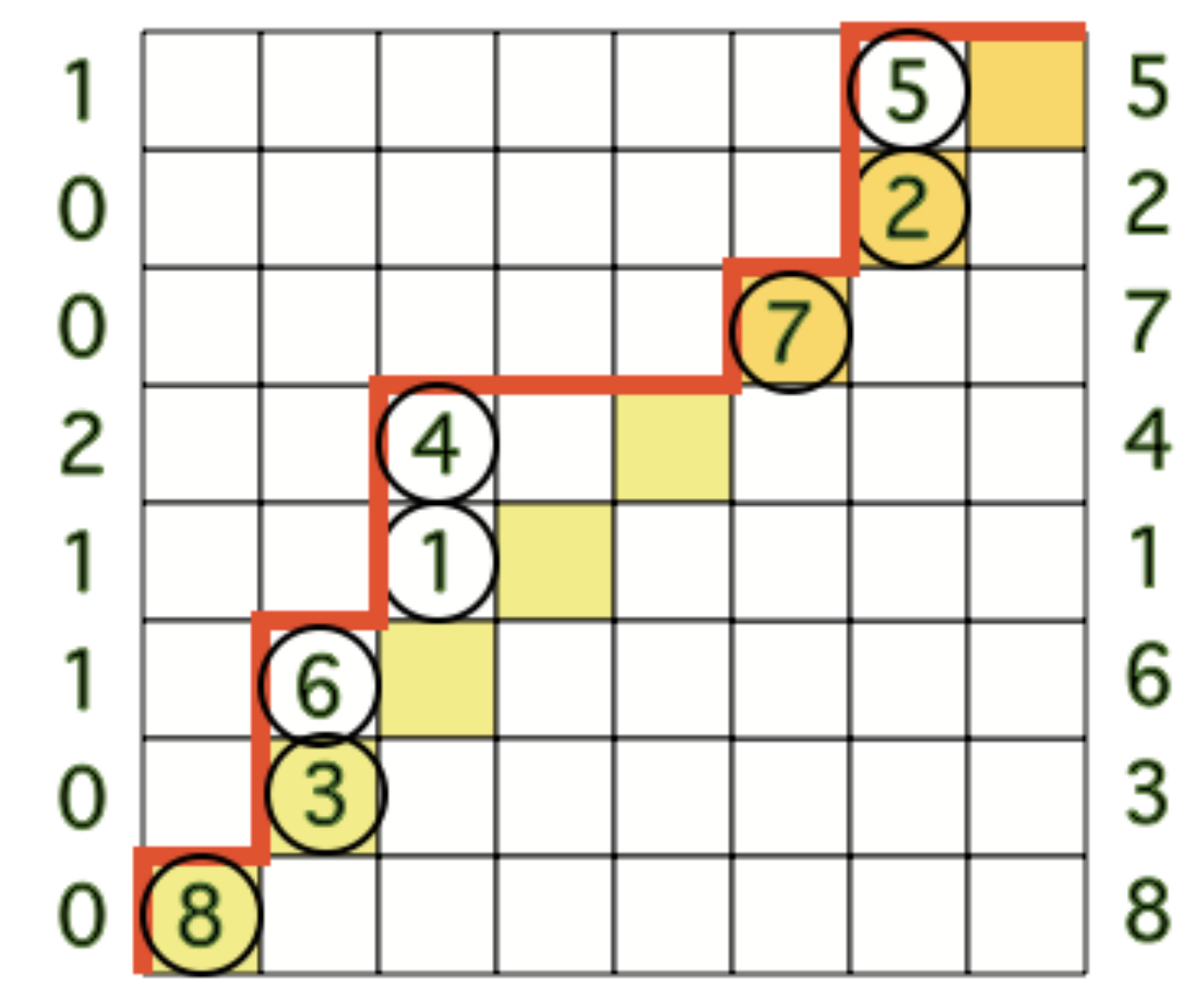}}}  
\bigsp \hbox{and} \bigsp
\vcenter{\hbox{ \includegraphics[width=1.3in]{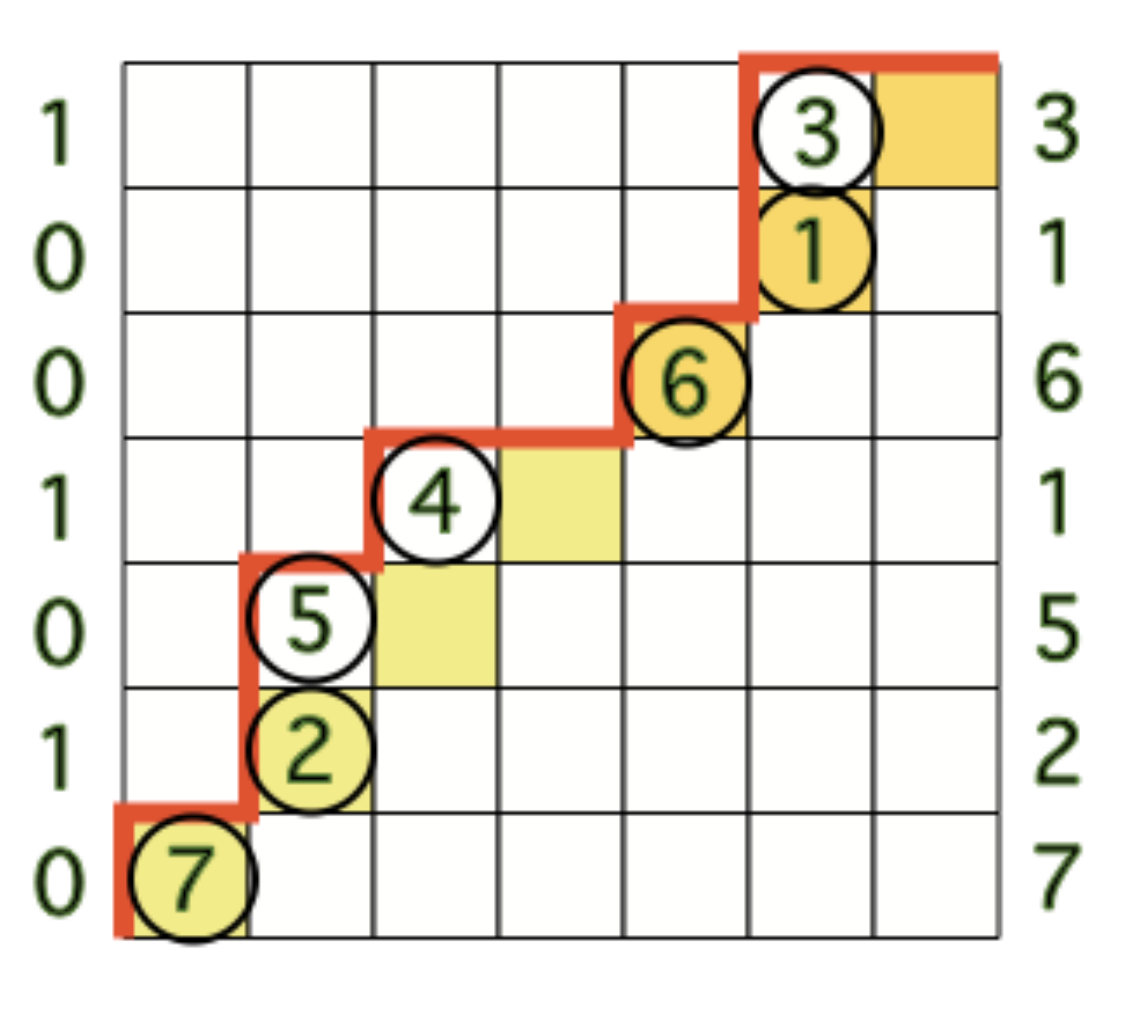}}}  
%\eqno 3.7
\label{eq3p7}
\end{equation}
Now in the first $PF$, the pairs $(3,7)$,  $(1,5)$ and $(6,2)$,
are the only ones contributing to the dinv and the sum of the area numbers is $5$,
so its classical weight is $t^5q^3$. Similarly,
the pairs contributing to the dinv on the $PF$ on the right are 
$(2,6)$,  $(5,1)$ and $(4,1)$ and the area numbers add to $3$, so its 
classical weight is $t^3q^3$. The latter is not the same as what comes out of \eqref{eq3p4}.
The area is OK but the dinv in not. The calculation in \eqref{eq3p5} 
thus asserts that the ``{\it new dinv}'' should be $4$. 
Similarly, as we will show in a moment, the calculation in \eqref{eq3p6} yields that
the new dinv of the $PF$ on the left of \eqref{eq3p7} should be $4$ again. In fact,
it turns out that none of the $8$ parking functions we obtain by inserting an extra $1$ 
in the reduced tableau on the right of \eqref{eq3p4} have area $5$ thus the last term in \eqref{eq3p6} can only be produced by the $PF$ on the left of \eqref{eq3p7}.  We give below the $8$ above mentioned reduced tableaux with the extra $1$ shaded
\begin{equation}
\vcenter{\hbox{ \includegraphics[width=5in]{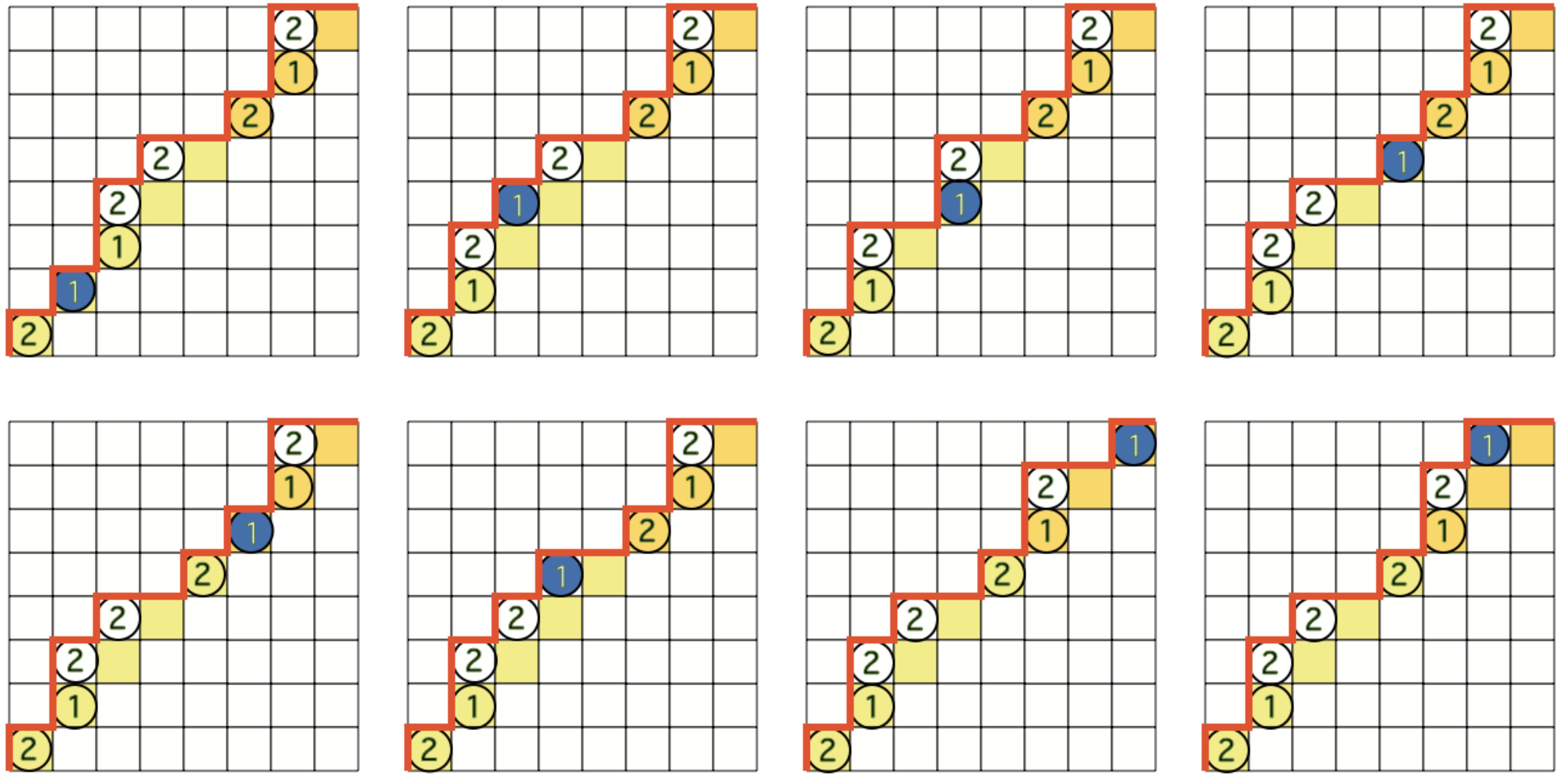}}}  
%\eqno 3.8
\label{eq3p8}
\end{equation}
Therefore the reduced tableaux of the family $\CAP_{3} ([3,2])$ are $9$ altogether,
as predicted by \eqref{eq3p6}, namely  the $8$ above together with 
the tableau on the left of \eqref{eq3p4}. Computing their classical weight and summing
gives
$$
\sum_{PF\in \CAP_{3} ([3,2])} t^{area(PF)}q^{dinv(PF)}
\ses t^3(q^4+2q^5+2q^6)+t^4(q^3+q^4+q^5)+t^5 q^3
$$
the first $8$ terms from \eqref{eq3p8} and the last from the left tableau in \eqref{eq3p4}.
As we see this is not quite the same polynomial as in \eqref{eq3p6}. Note that the area again  works but the  classical dinv does not!.
\sas

For a while in our investigation this appeared to be a challenging puzzle.
The  discovery of the recursion of Theorem \ref{th:IIp1} completely solved this puzzle
but, as we shall see, it created  another puzzle. 
\sas

Let us have a  closer look at \eqref{eqIIp16}, namely
the identity
\begin{align}
\LL \Delta_{h_J}\BC_{p_1}\BC_{p_2}\cdots \BC_{p_k} 1\scs  e_n \RR 
&\ses
t^{p_1-1}q^{k-1}
\LL \Delta_{h_{J-1}} \BC_{p_2}\cdots \BC_{p_k}\BB_{p_1} 1\scs  
 e_{n} \RR \nonumber
 \\
 &\bigsp\bigsp
 \sps \chi(p_1=1)\LL \Delta_{h_{j}}\BC_{p_2}\cdots \BC_{p_k} 1\scs   e_{n-1} \RR 
\label{eq3p9}
%\eqno 3.9 
\end{align}
Setting for a composition $p=(p_1,p_2,\ldots p_k)\models n$
$$
\PI_J(p)\ses \sum_{PF\in \CAP_J(p)}t^{area(PF)}q^{ndinv(PF)}
$$
our conjecture, together with the identity
$$
\BB_{p_1} 1\ses e_{p_1}= \sum_{(q_1,q_2,\ldots ,q_\ell)\models p_1}
\BC_{q_1}\BC_{q_2}\cdots \BC_{q_\ell} 1
$$
proved in \cite{HHLRU},  translates \eqref{eq3p9} into the recursion
\begin{equation}
\PI_J([p_1,p_2,\ldots p_k])\ses   t^{p_1-1}q^{k-1}
\sum_{r\models p_1}\PI_{J-1}([p_2,\ldots p_k,r])\sps
 \chi(p_1=1)\PI_J([p_2,\ldots p_k])
%\eqno 3.10 
\label{eq3p10}
\end{equation}
where the symbol $[p_2,\ldots p_k,r]$ represents the concatenation of the 
compositions  $(p_2,\ldots p_k)$ and $r$. This strongly suggests
what should recursively happen  to the new weight of our parking functions
by the removal of a single (appropriate) car. That is, if the chosen car is $small$
there should be a loss of area of $p_1-1$ and a loss of ndinv of $k-1$ and if
the chosen car is $big$ no loss of any kind. 

Starting from this observation and further closer analysis of \eqref{eq3p9} led us to the following
recursive algorithm for constructing ``$ndinv$''.

This is best described by working with the corresponding reduced tableaux.
To begin it will be convenient to start by decomposing each $red(PF)$ into sections
corresponding to the parts of the given composition. To be more precise it is best
viewing our two line arrays as unions of {\it vertical dominos}. For instance 
the $red(PF)$ below, which is none other than the {\it minimal ones} obtained from
 the Dyck path on the right
 $$
\vcenter{\hbox{ \includegraphics[width=2in]{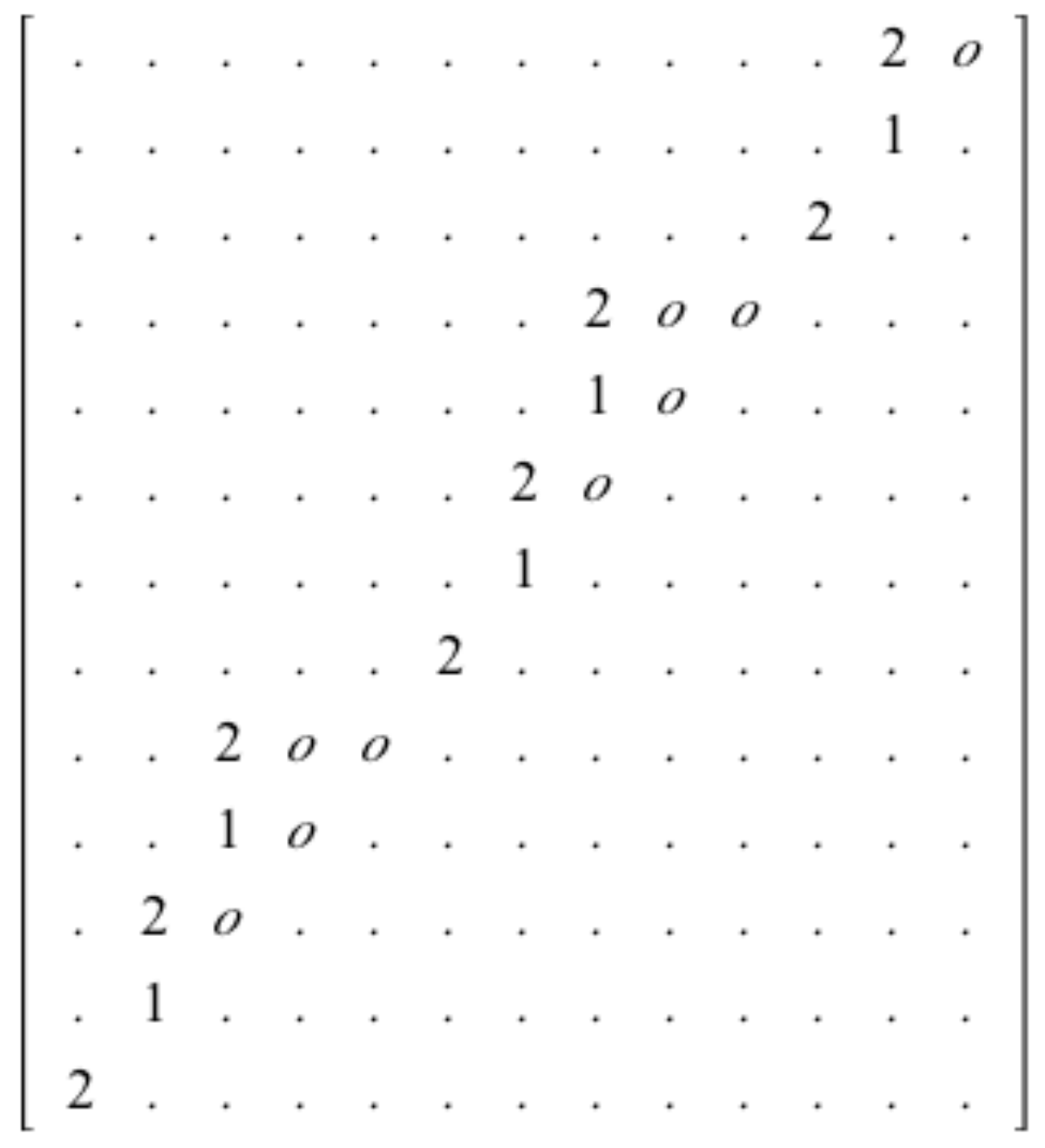}}}
\bigsp
\vcenter{\hbox{ \includegraphics[width=1.2in]{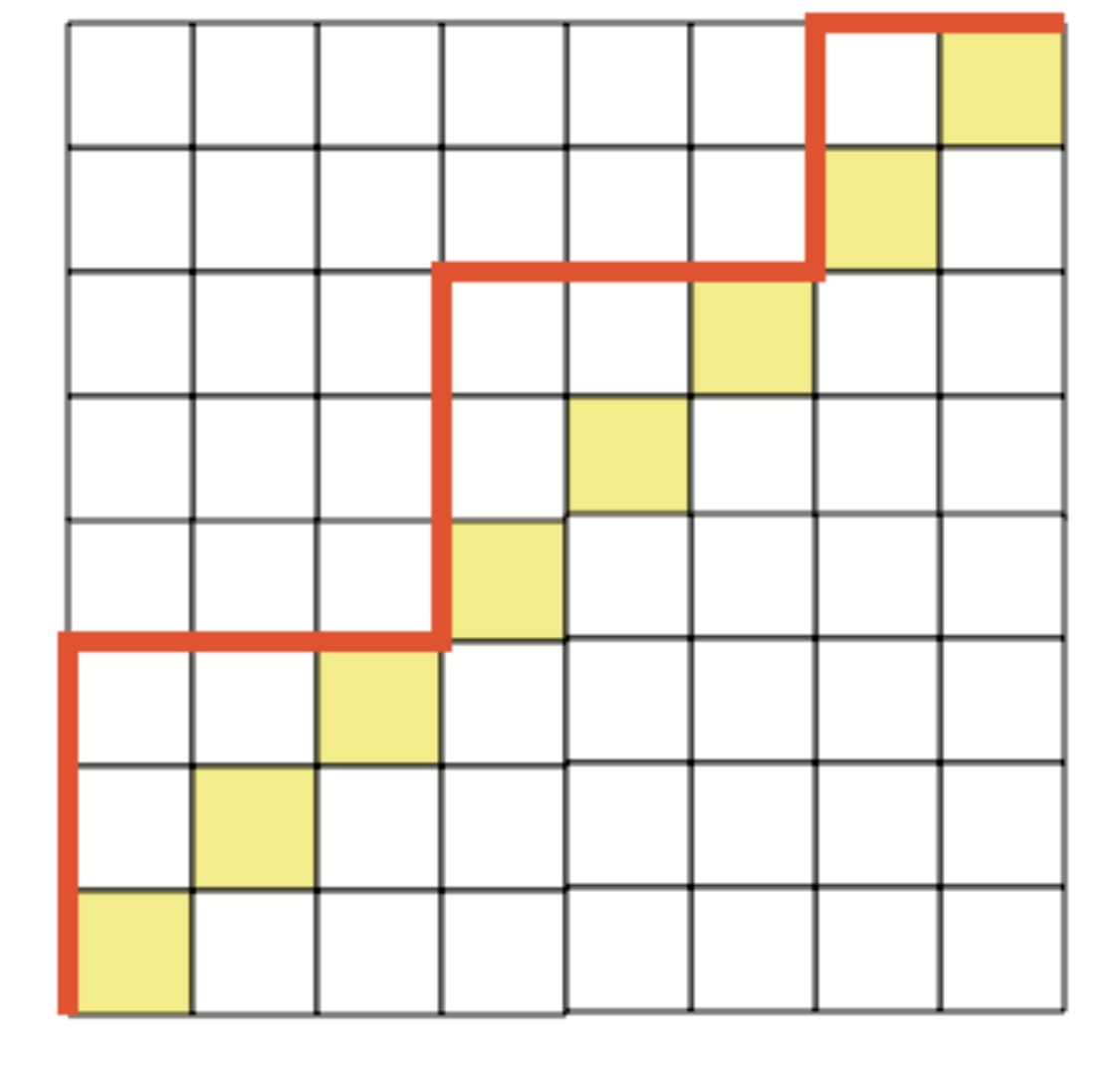}}}  
$$
will be viewed as the sequence of dominos 
\begin{equation}
\vcenter{\hbox{ \includegraphics[width=4in]{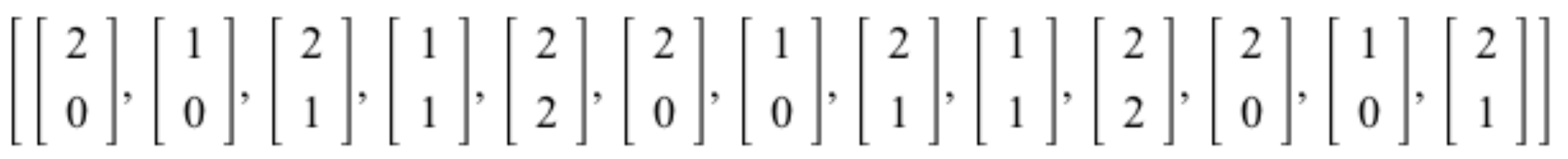}}}  
%\eqno 3.11
\label{eq3p11}
\end{equation}
Thus the corresponding $PF$ belongs to  the family $\CAP_5([3,3,2])$
and as such will be divided into $3$ sections, one for each part of $[3,3,2]$.
To do this we simply cut the sequence in \eqref{eq3p11} before each domino $[{2\atop 0}]$
obtaining the three sections 
\begin{equation}
\vcenter{\hbox{ \includegraphics[width=4.5in]{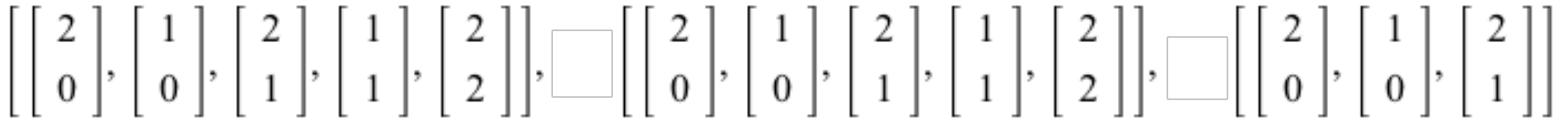}}}   
%\eqno 3.12
\label{eq3p12}
\end{equation}
Since  here  $p_1=3>1$, \eqref{eq3p9} suggests that we should remove a $1$ from 
the first section, then process it somewhat to cause a loss of dinv of $2$ ( $=k-1$),
and loss of area $2$ ($=p_1-1$). Taking a clue from the classical dinv, we can see that
the first small car in the corresponding $PF$ would contribute a unit to the classical  
dinv with the big cars to its right in the main diagonal. The latter  of course 
correspond to the dominoes  $[{2\atop 0}]$ that begin each of the following sections.
Thus the desired loss of dinv can be simply obtained by bodily moving the first section
to the end, and removing the  $[{1\atop 0}]$ obtaining  
\begin{equation}
\vcenter{\hbox{ \includegraphics[width=4.5in]{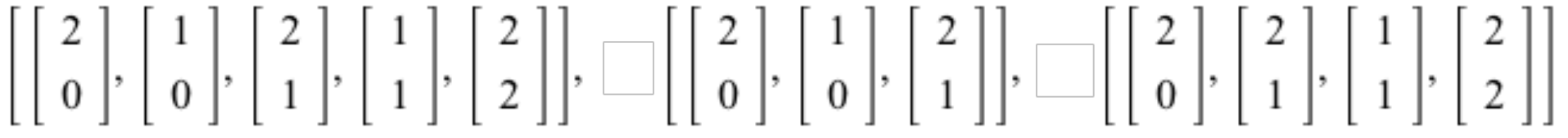}}} 
%\eqno 3.13
\label{eq3p13}
\end{equation}
 We may thus view that
the removed domino contributed a unit to dinv for each domino 
 $[{2\atop 0}]$ to its right.  But we still have not accounted for the loss of area
 and worse yet we will now have a big car on top of a big car.
 Since \eqref{eq3p9} tells that the loss of area   should be
$p_1-1$ then it must be  equal to the number of big cars in the moved section,
minus one. This means that we can fix both problems  by making the domino replacements 
$[{2\atop 1}]\ess \RA\ess  [{2\atop 0}]$  and $[{2\atop 2}]\ess \RA\ess  [{2\atop 1}]$,
obtaining 
\begin{equation}
\vcenter{\hbox{ \includegraphics[width=4.5in]{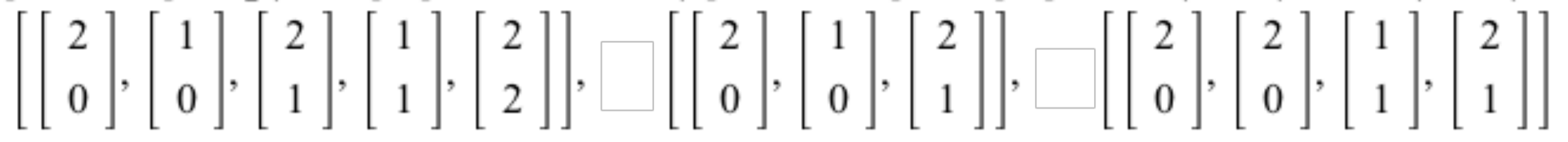}}}  
\ess\ess 
%\eqno 3.14
\label{eq3p14}
\end{equation}
but that creates a new problem, since the succession $[{2\atop 0}],[{1\atop 1}]$
would put a small car on top of a big car. We will fix this final problem 
by simply switching the $1$ with the $2$ obtaining   
\begin{equation}
\vcenter{\hbox{ \includegraphics[width=4.5in]{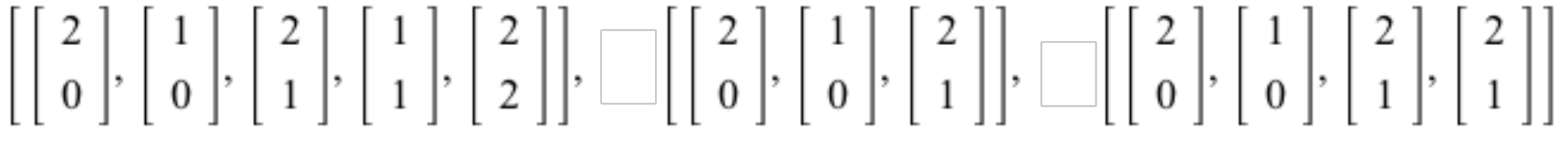}}}  
\ess\ess 
%\eqno 3.15
\label{eq3p15}
\end{equation}
which  gives the domino sequence of the $red(PF)$ below
$$
\vcenter{\hbox{ \includegraphics[width=1.8in]{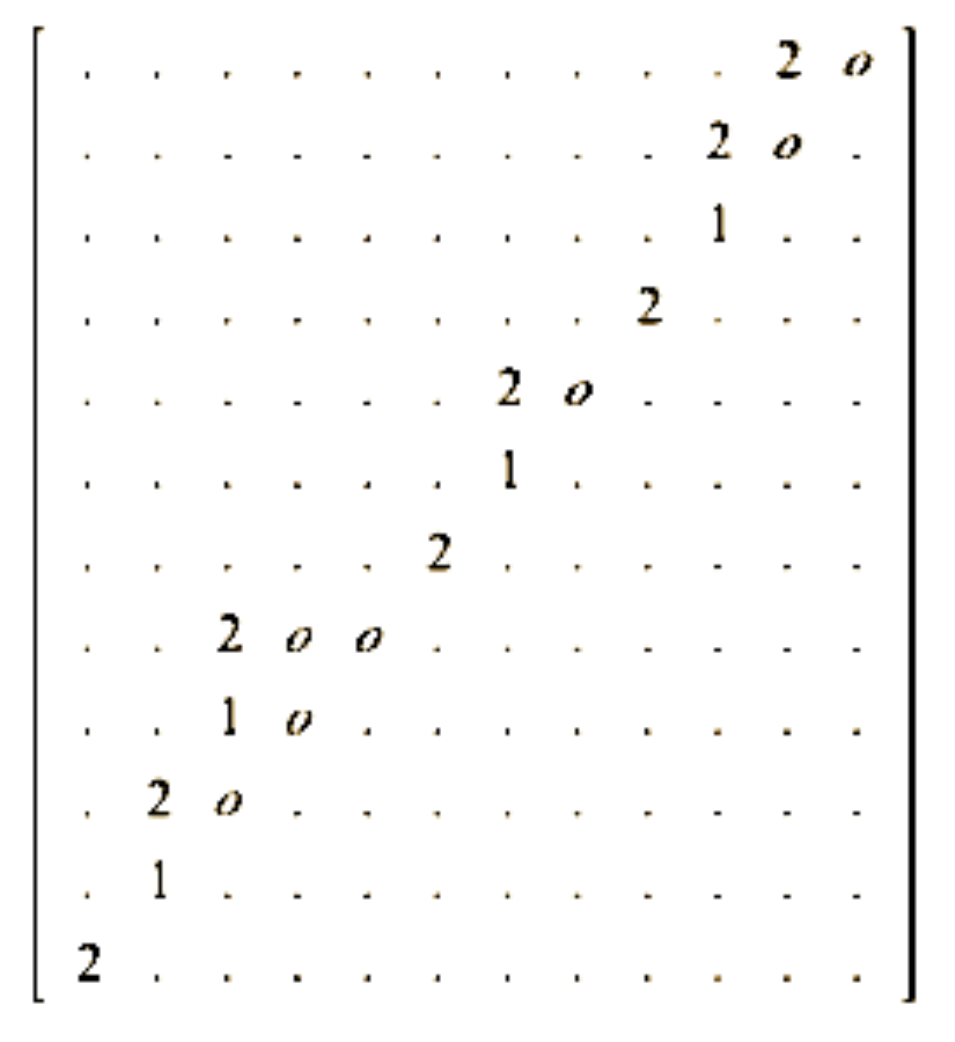}}}
\bigsp
\vcenter{\hbox{ \includegraphics[width=1.2in]{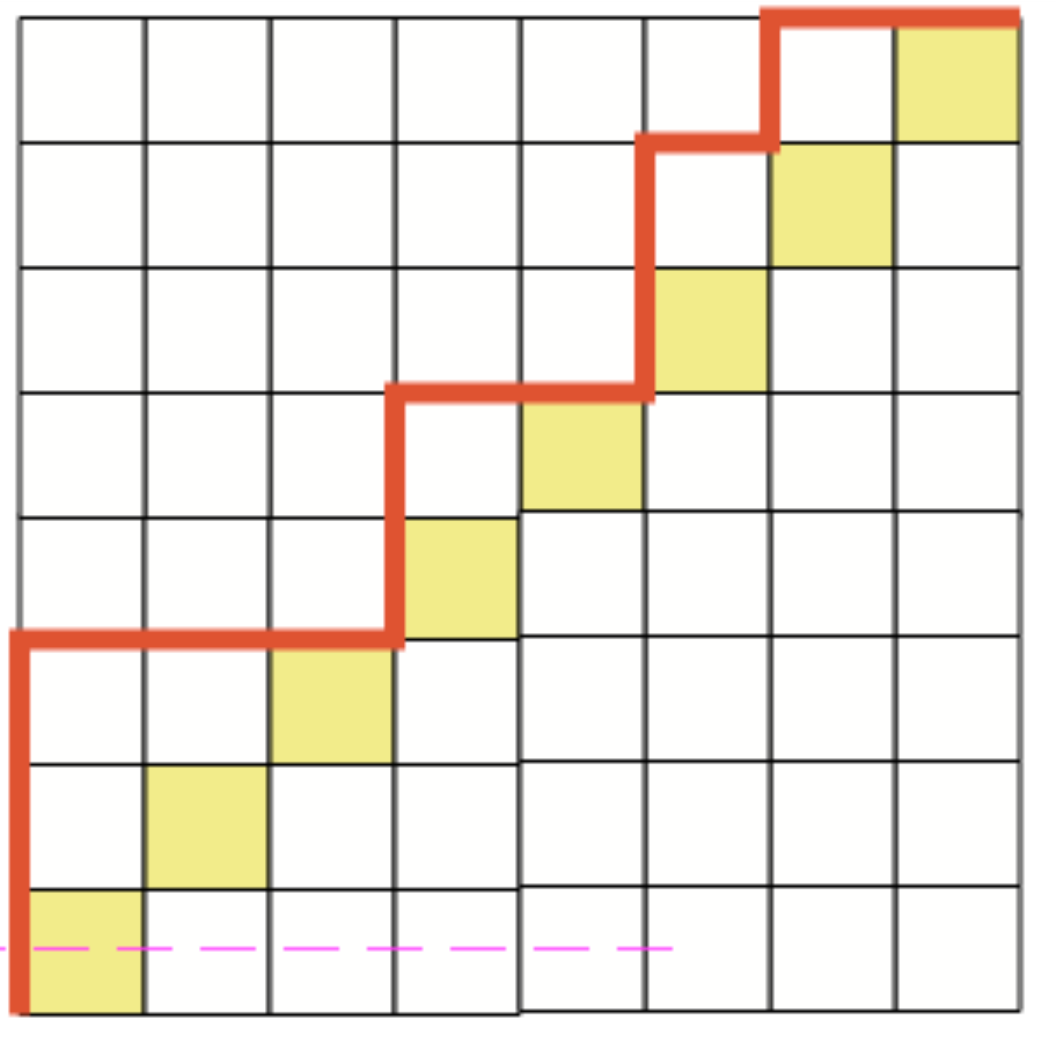}}}
$$
on the right of which we have depicted the Dyck path supporting the big cars.
\vfill

In the case that $p_1=1$ there will be only one  big car in the first section and
if there are small cars they all must be on the main diagonal. In this
case we can process the first section  as we did for $p_1>1$. If
there are no small cars then the first section consists of the single domino
$[{2\atop 0}]$ and \eqref{eq3p9} suggests that we should simply remove it 
with no further ado.
\sas

To carry out our definition of ``ndinv'' rigorously and in full generality,
we will break our argument into three separate steps. In the firs step
we use the ideas stemming from the above example to construct
a bijection 
\begin{equation}
\Phi \ssp :\ssp \CAP_J([p_1,p_2,\ldots p_k])
\ess\Longleftrightarrow\ess
\begin{cases}
\displaystyle\bigcup_{q\models p_1}\CAP_{J-1}([p_2,\ldots p_k,q]) & \hbox{ if }p_1>1\\
\CAP_{J-1}([p_2,\ldots p_k,1])\ssp \oplus \ssp 
\CAP_{J}([p_2,\ldots p_k])&\hbox{ if }p_1=1 
\end{cases}
%\eqno 3.16
\label{eq3p16}
\end{equation}
In the second step we define ``ndinv`` by setting for each $PF\in \CAP_J([p_1,p_2,\ldots p_k])$
\begin{equation}
ndinv(PF)\ses \begin{cases}
k-1+ndinv(\Phi(PF)) &\hbox{ if }J>0\\
0 &\hbox{ if }J=0 
\end{cases}
%\eqno 3.17
\label{eq3p17}
\end{equation}
From step 1 and step 2 it will follow that  the polynomials $\Pi_J([p_1,p_2,\ldots ,p_k])$
satisfy the same recursion as the polynomials 
$\LL\DD_{h_J}\BC_{p_1}\BC_{p_2}\cdots \BC_{p_k}1\scs e_n\RR$.

In the third step we establish the equality in \eqref{eq3p2}  by verifying the equality in the base
cases.
\sas

\noindent
In our first step, starting with a $PF\in \CAP_J([p_1,p_2,\ldots p_k])$
we construct $\Phi(PF)$ by the following procedure.

\begin{itemize}
\item {\it Cut the domino sequence of $red(PF)$ into sections starting at the dominos
$[{2\atop 0}]$}
\end{itemize}

\begin{enumerate}

\item If the first  section does not contain a domino  $[{1\atop 0}]$

\begin{itemize}
\item  {\it remove its only domino $[{2\atop 0}]$
from the sequence of dominos}
\end{itemize}

\item If the first  section  contains a domino   
$[{1\atop 0}]$, work on the first section as follows

\begin{itemize}
\item {\it Remove its first domino $[{1\atop 0}]$ }

\item {\it For each   (but the first)  domino $[{2\atop a}]$
make the replacement $[{2\atop a}]\ssp \RA \ssp [{2\atop a-1}]$}

\item {\it if  adjacent  pairs $ [{2\atop a-1}] [{1\atop a}]$ are  created 
make the replacements  
$ [{2\atop a-1}] [{1\atop a}]\ssp \RA \ssp [{1\atop a-1}] [{2\atop a}]$}

\item {\it Cycle the modified first section  to the end of the sequence
of dominos }
\end{itemize}
\end{enumerate}
\noindent  In any case we let   $PF'$ be the parking function corresponding to the
resulting domino sequence and set  
$$
\Phi(PF)=PF'.
$$   

It is clear that $\Phi$ maps the left hand side of \eqref{eq3p16} into the right hand side.
To show that $\Phi$ is a bijection we need only show that the procedure above 
can be reversed to reconstruct  $PF$ from  $PF'$  for any $PF'$ in the
right hand side of \eqref{eq3p16}. We will outline the salient steps of the reversed procedure.
\sas

Note first that since our target $PF=\Phi^{-1}(PF')$ is to be in $\CAP_J([p_1,p_2,\ldots p_k])$ we already know the diagonal composition of the Dyck path of the big
cars of $PF$. Thus we can proceed as follows
\sas

\begin{enumerate}
\item Say $PF'\in \CAP_{J}([p_2,\ldots p_k])$ (that will only occur when $p_1=1$)

\begin{itemize}
\item {\it Then $PF$ is the parking function obtained by prepending 
$[{2\atop 0}]$ to the domino sequence of $PF'$.}
\end{itemize}
 
\item Say $PF'\in \CAP_{J-1}([p_2,\ldots p_k,1])$ (that will only occur when $p_1=1$)

\begin{itemize}
\item  {\it  Then $PF$ is the parking function obtained by inserting 
$[{1\atop 0}]$ immediately after the first  $[{2\atop 0}]$ in the last section of  $PF'$,
then cycle back the last section to be the first in the domino sequence.}
\end{itemize}

\item Say $PF'\in \CAP_{J-1}([p_2,\ldots p_k,q])$ for a $q\models p_1-1>0$

\begin{itemize}
\item {\it Let $last(PF')$ be the domino sequence obtained  
by removing the first $k-1$ sections from the  domino sequence of $PF'$.}
\item {\it Modify  $last(PF')$ by
inserting  a  $[{1\atop 0}] $ immediately after its first
$[{2\atop 0}]$.}
\item {\it For $a\ge 1$ replace, in  $last(PF')$,  each pair  $  [{1\atop a-1}] [{2\atop a}]$
by the pair $[{2\atop a}] [{1\atop a}]$.}
 \sas

\item[] (note that for this to put a big car on top of a big car we must have   a  $ [{2\atop a-1}]$ preceding the $[{1\atop a-1}]$, but that $ [{2\atop a-1}]$
 will also be replaced either by this step or by the next steps)
  \sas
  
\item {\it
For $a\ge 1$ replace each $[{2\atop a}]$ 
preceded by a $[{1\atop a}]$ in $last(PF')$ by $[{2\atop a+1}]$}

\item {\it  Replace each  $ [{2\atop 0}]$, except the first by a $ [{2\atop1}]$}
 
\item[]
(note if a replaced $ [{2\atop 0}]$ is preceded by a $ [{2\atop 0}]$ then that
 $\ssp $ $ [{2\atop 0}]$ itself will also be replaced by $ [{2\atop1}]$) 

\item {\it
The modified $last(PF')$ followed by the the first $k-1$ sections of $PF'$
gives then the domino sequence of our target $PF$.}
\end{itemize}
\end{enumerate}

\noindent
This completes our proof that $\Phi$ is bijective. 
\sas

Since $\Phi$ moves   EAST, by one cell, $p_1-1$  big cars it causes 
  a loss of area equal to $p_1-1$. Thus the definition
in \eqref{eq3p17} combined by the bijectivity of $\Phi$ proves  the recursion in \eqref{eq3p10}.
\sas

It remains to show equality in the base cases which, in view of the definition in \eqref{eq3p17}
should be  characterized by the absence of small cars.

Now it easily seen, combinatorially,  that $\CAP_0([p])$
is an empty family except when all the components of $p$ are equal to
$1$. To see this note that it is only the presence of small cars that allows
the supporting Dyck path of  one of our $PF's$  to have columns 
of lengths $2$. But if all the columns are of length $1$, the ``area'' 
statistic  is $0$ and the Dyck path
supporting the big cars  can only have area sequence a string of $0's$.
But in this case the family reduces to a single parking function 
which consists of cars $1,2, \cdots ,n$ placed on the main diagonal from top to bottom.
Thus it follows from our definition  
of $\Pi_J(p)$  and \eqref{eq3p17} that
$$
\Pi_0([p_1,p_2,\ldots ,p_k])\ses
\begin{cases}
0 &\hbox{ if some }p_i>1\\
1 &\hbox{ if all }p_i=1
\end{cases}
$$
Since by definition  $\DD_{h_0}$ reduces to the identity operator,
 the equality for the bases cases results from the following fact
\sas

%Theorem 3.1}
\begin{theorem}
For $p=(p_1,p_2,\ldots ,p_k)\models n$ we have
\begin{equation}
\LL\BC_{p_1}\BC_{p_2}\cdots \BC_{p_k} 1\scs e_n\RR\ses 
\begin{cases}
0 &\hbox{ if some }p_i>1\\
1 &\hbox{ if all }p_i=1
\end{cases}
%\eqno 3.18
\label{eq3p18}
\end{equation}
\end{theorem}

\begin{proof}
Recall from \eqref{eqIIp2} that for any symmetric function $F[X]$ we have
\begin{equation}
\BC_a F[X]\ses (-{\TS{ 1\over q }})^{a-1} F\big[X-\TS{1-1/q \over z}
\big]\displaystyle
\sum_{m\ge 0}z^m h_{m}[X]\, \Big|_{z^a},
%\eqno 3.19
\label{eq3p19}
\end{equation}
In particular it follows that for any Schur function $s_\la$ we have
$$  
\BC_a s_\la[X] \ses (-{\TS{ 1\over q }})^{a-1}
\sum_{\mu\con \la}s_{\la/\mu}[X]s_\mu[1/q-1] 
\displaystyle
h_{a+|\mu|}[X] 
$$
This gives for $a+|\la |=n$
\begin{equation}
\LL
\BC_a s_\la[X] \scs e_n
\RR\ses 
(-{\TS{ 1\over q }})^{a-1}
\sum_{\mu\con \la}s_\mu[1/q-1] \LL s_{\la/\mu}  
h_{a+|\mu|} \scs e_n\RR
%\eqno 3.20
\label{eq3p20}
\end{equation}
and the Littlewood-Richardson rule gives
$$
\LL s_{\la/\mu}  
h_{a+|\mu|} \scs e_n\RR\ses \LL s_{\la/\mu}  
\scs h_{a+|\mu|}^\perp e_n\RR\ses 0
$$
unless $a+|\mu|=1$, Thus for $a\ge 1$ \eqref{eq3p20} reduces to
\begin{equation}
\LL
\BC_a s_\la[X] \scs e_n
\RR\ses 
(-{\TS{ 1\over q }})^{a-1}
 \LL s_{\la}  
h_{a} \scs e_n\RR
\ses
\begin{cases}
1 &\hbox{ if }a=1\hbox{ and } \la=1^{n-a}\\
0 &\hbox{ otherwise}
\end{cases}
%\eqno 3.21
\label{eq3p21}
\end{equation}
Since 
$$
\BC_a 1\ses 
(-{\TS{ 1\over q }})^{a-1}
h_a
$$
the first case of \eqref{eq3p18} follows immediately from \eqref{eq3p21}.  
On the other hand even when all the $p_i$ are equal to $1$ 
in successive applications of $\BC_1$ only the term 
corresponding to  $s_{1^m}$ in the Schur function expansion of
$\BC_1^m 1$ will  survive in the scalar product
$$
\LL\BC_1  \BC_1^m 1\scs e_{m+1}\RR
$$
since from \eqref{eq3p21} it follows that $\BC_1^m 1 \Big|_{s_{1^m}}=1$,
the second case of \eqref{eq3p18} is also another consequence of \eqref{eq3p21}. 

This completes of proof of \eqref{eq3p18}. This  was the last fact we need to 
establish the equality in \eqref{eqIIp17}.  
\end{proof}

%Remark 3.1}
\begin{remark}
As we already  mentioned, our definition of $ndinv$ creates another puzzle.
Indeed, the classical $dinv$ can be immediately computed from the geometry of the
parking function or directly from \eqref{eqIIp6} which  expresses it explicitly
in terms  of  the two line array representation.
For this reason we made a particular effort to obtain a non recursive
construction of $ndinv$ and in the best scenario derive form it an explicit formula
similar to \eqref{eqIIp6}.
However our efforts yielded only a partially non-recursive construction.
In our original  plan of writing we decided to include this further result 
even though in the end it yields a more complex algorithm for computing 
$ndinv$ than from the original recursion.  This was in the hope that
our final construction may be conducive to the discovery 
of an  explicit formula. It develops that during  the  preparation of
this manuscript a new and better reason emerged for the inclusion
of our final construction. It turns out that Angela Hicks and Yeonkyung Kim
have very recently succeeded  in discovering the desired explicit formula
by a careful analysis of the combinatorial  identities we are about to present.
The results  of Hicks-Kim will appear in a separate  publication  \cite{HicksKim}. 
\end{remark}

For our less recursive  construction of  $ndinv$ it will be convenient to
make a few changes in the domino sequences. To begin, we shall use the actual car numbers at the top of the dominos rather than $1$ or $2$. 
We do this, so  that we may 
refer to individual dominos by their car as the corresponding area number
on the bottom is being changed. But now, to distinguish  big cars from 
small cars we must in each case specify the number $J$ of small cars.
Secondly, we will have sections
end with a big car, rather than begin with a big car. This only requires,
moving  the initial big car to the end of the domino sequence. 
For example,  the parking function below whose domino sequence was 
given in \eqref{eq3p11} has $J=5$ thus cars $1,2,3,4,5$ are $\underline{ small}$  and   $5,6\ldots,13$ 
are $\underline{big}$.
\begin{equation}
\vcenter{\hbox{ \includegraphics[width=2in]{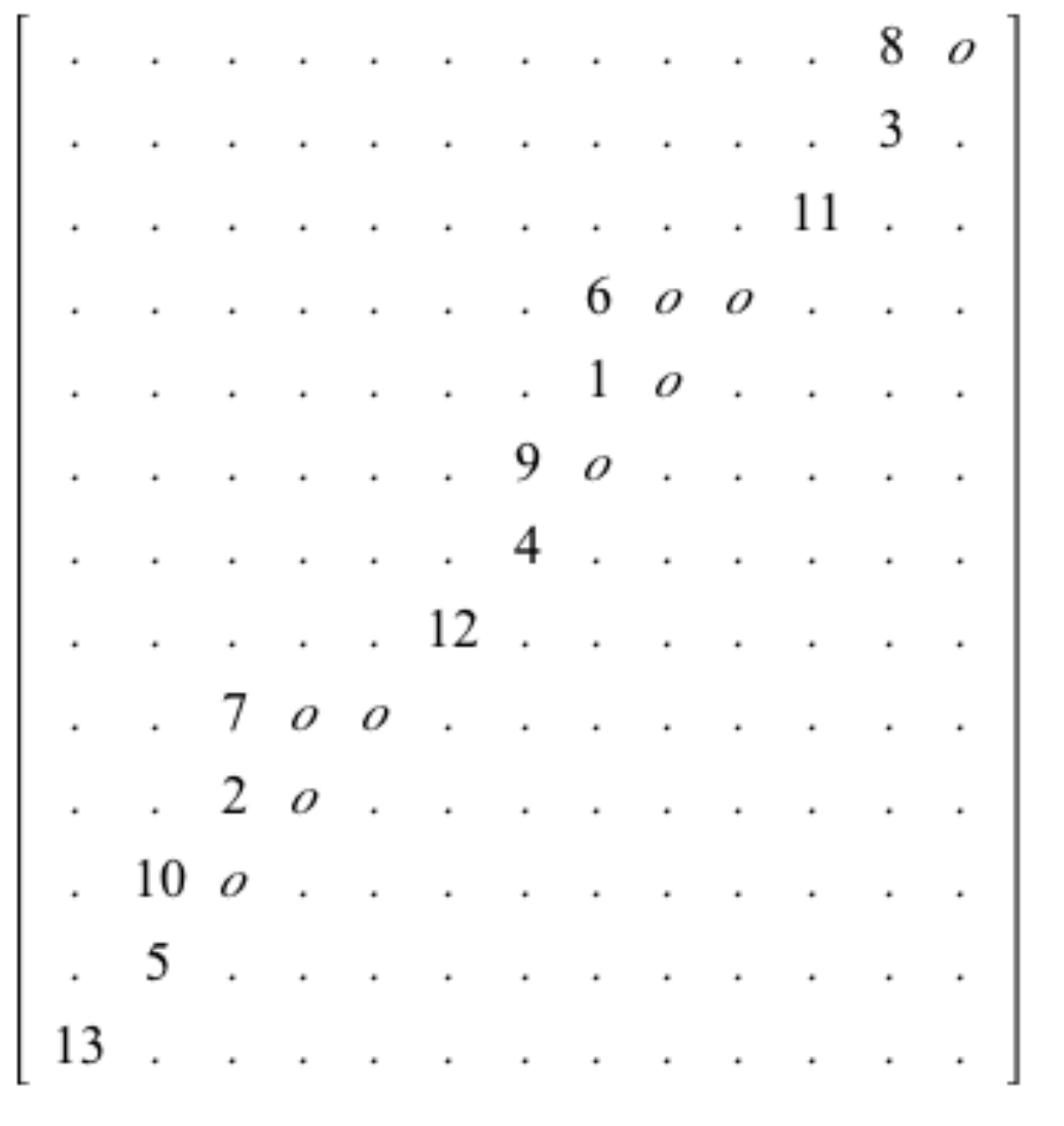}}}
%\eqno 3.22
\label{eq3p22}
\end{equation}
its  domino sequence is now
\begin{equation}
\vcenter{\hbox{ \includegraphics[width=4.5in]{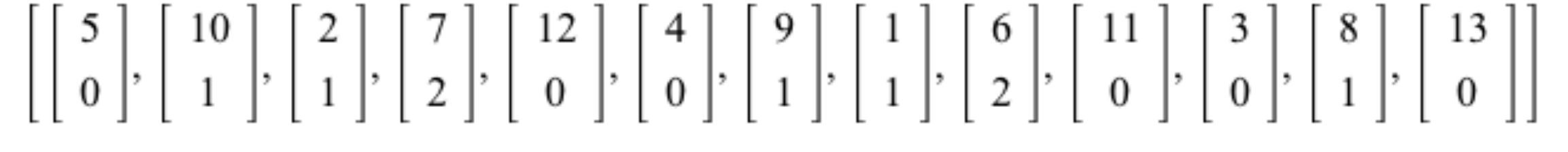}}}  
%\eqno 3.23
\label{eq3p23}
\end{equation}
and its decomposition into sections is  as shown below 
\begin{equation}
\vcenter{\hbox{ \includegraphics[width=5in]{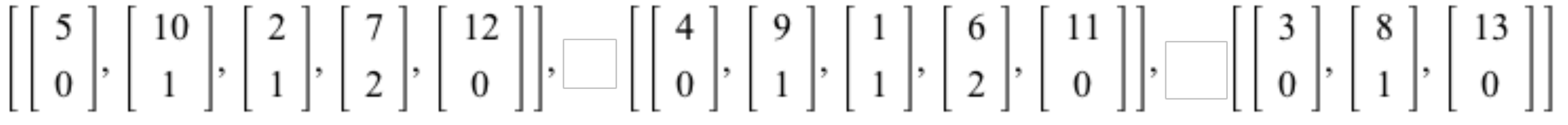}}}  
%\eqno 3.24
\label{eq3p24}
\end{equation}

For convenience we will now need  to use the symbols ``$[{b\atop a}]$'' 
and ``$[{s\atop a}]$'' to respectively represent ``big'' and ``small '' car
dominos.
To motivate our second construction of $ndinv$ we will begin by 
modifying  our first construction to adapt to these new domino sequences.

%\item
{ (1)}
{\it The recursive construction will now 
consist of as many steps as there are dominos in the domino sequence}

%\item
{ (2)}
{\it At each step the first  domino of the first section is removed}

%\itemitem
{(a)}{\it  when we remove an $[{s\atop 0}]$,  the section is
cycled  to the end after it is processed  as before}

%\itemitem
{(b)}{\it  when  we remove a $[{b\atop 0}]$, it is because the
section 
consisted of a single big car domino.
 }

%\item
{ (3)}{\it The  removal of  an $[{s\atop 0}]$ contributes to $ndinv$   the number of $[{b\atop 0}]$'s to its right minus one.}

There are a few observations to be made about the effect of the cycling process.  To begin note that when the domino sequence consists of
a single section, no visible cycling occurs.  However, even  in this case, for accounting purposes, it is convenient to consider all of its  dominos to have been cycled. With this provision, each domino in the original domino sequence will have an associated cycling number $c$ that counts the number of times it has been cycled  before it is removed.

Based on these observations, a step by step study of our recursive construction of  $ndinv$ led us to the following somewhat less recursive 
algorithm. It consists of two stages. In the first stage, the domino sequence
is doctored and wrapped around a circle to be used in the second stage.
The second stage uses circular motion to mimic  
the cycling of sections that takes place in the recursive procedure. 
To facilitate the understanding of  the resulting algorithm
we will illustrate each stage by applying it to the parking function in \eqref{eq3p22}. 
More precisely we work as follows
\sas
\noindent
 \hskip .2in{\bf Stage I }
 
\begin{itemize}
\item {\it Move  each  $[{s\atop a}]$ in the domino sequence
$a$ places to its left and increase the area number by $1$ of each domino
$[{b\atop a}]$ that is being by-passed}
\end{itemize}

\centerline{ (for instance the domino sections in \eqref{eq3p24} become)}

\begin{equation}
\vcenter{\hbox{ \includegraphics[width=5in]{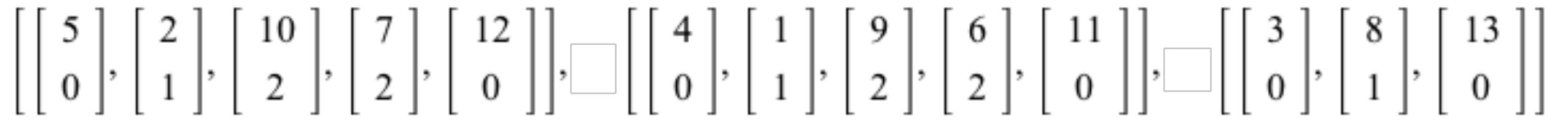}}}  
%\eqno 3.25
\label{eq3p25}
\end{equation}

\begin{itemize}
\item {\it Next wrap the resulting sequence clockwise 
around a circle with positions marked  by a ``$\circ$''}
\end{itemize}

\begin{equation}
\vcenter{\hbox{ \includegraphics[width=2.2in]{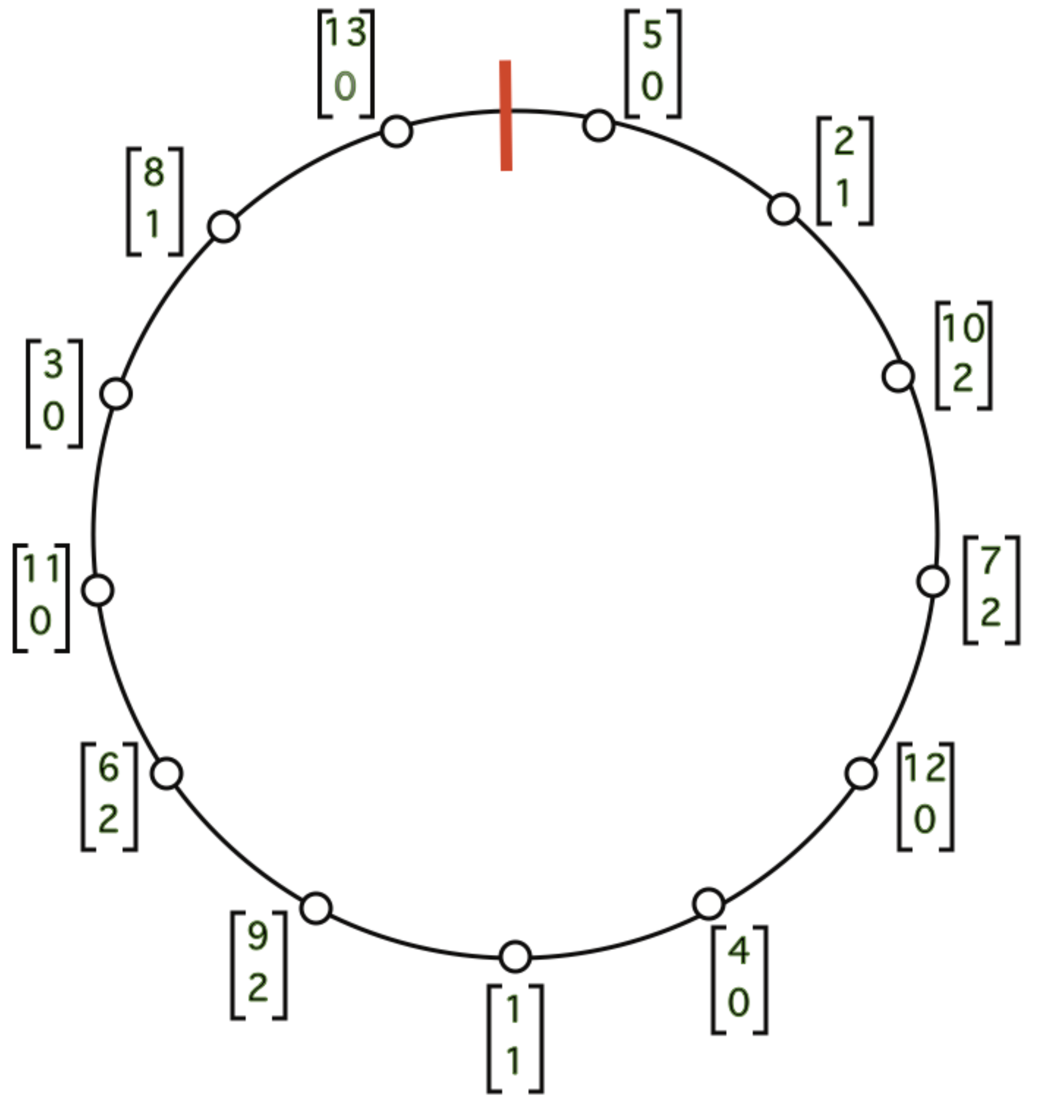}}} 
\ess\ess  \Longrightarrow  \ess\ess 
\vcenter{\hbox{ \includegraphics[width=2.2in]{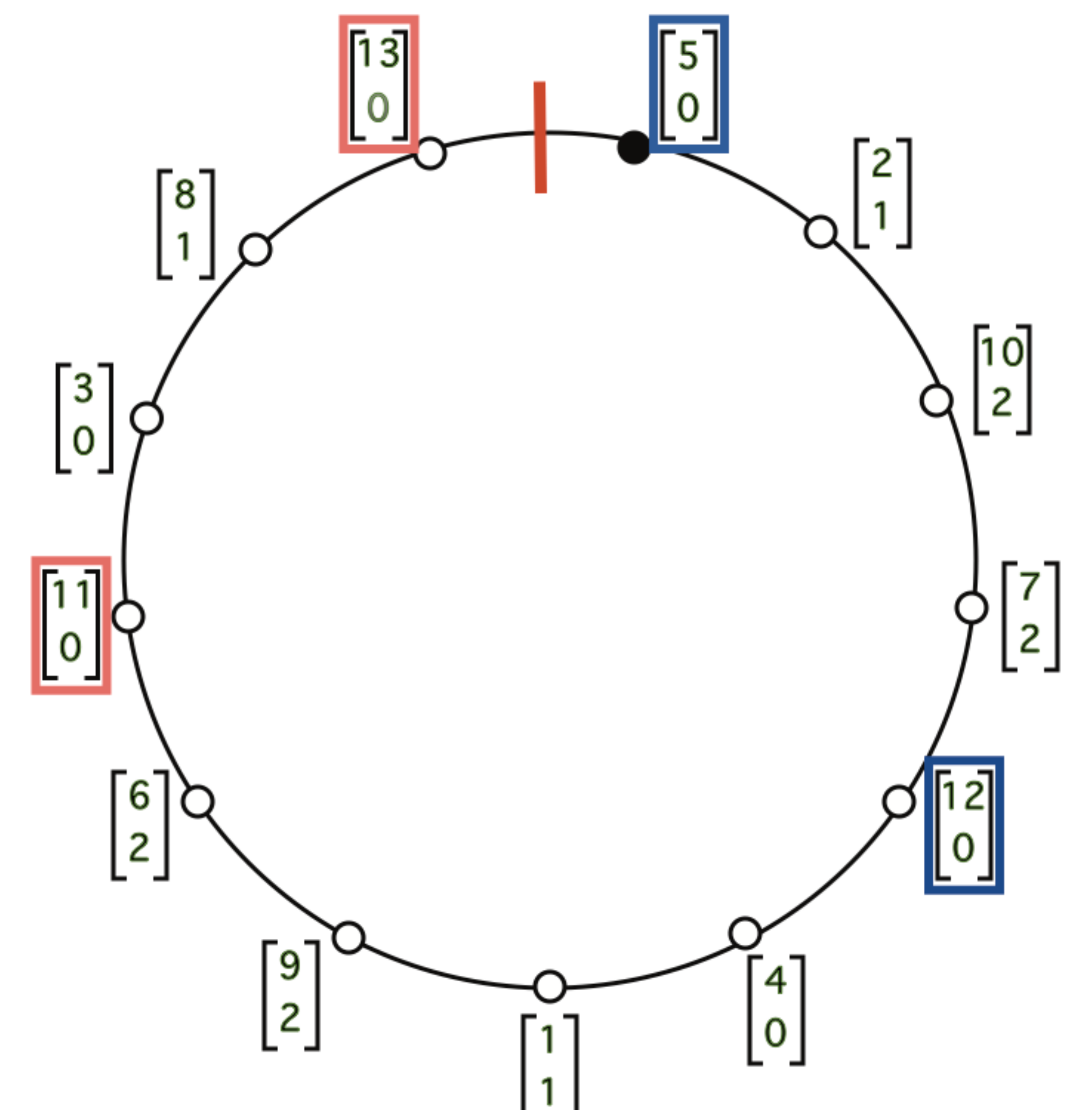}}} 
%\eqno 3.26
\label{eq3p26}
\end{equation}

\centerline{(we also place a  $bar $ ``$|$'' to separating  beginning and ending dominos )}
 
 \centerline{(the $\circ$'s will be successively changed to $\bullet$'s
 during the second stage) }
 \sas
 \noindent
 \hskip .2in {\bf Stage II }
 
\begin{itemize}
\item {\it set $ndinv=0$ and set the auxiliary parameter $c$ to $1$}
   
\item {\it mark the first domino by changing its ``$\circ$''
 to a ``$\bullet$''.  }

\item {\it cycling clockwise  from the first domino to the $bar $
  find the first  $[{b\atop 0}]$,  call it ``$endsec$''}
  
\item {\it cycling clockwise from $endsec$ to the $bar $
add $1$ to   $ndinv$ each time we meet a  $[{b\atop 0}]$.}
\end{itemize}

\centerline{ (on the right in \eqref{eq3p26} we have  darkly boxed the first domino and the $endsec$)} 
\centerline{ (and  lightly boxed the two $ndinv$ contributing  big car dominos)}

{\it  While there is a domino that has not been marked   repeat the following steps}
 
\begin{itemize}
\item {\it cycling clockwise from the last $endsec$
 mark the first unmarked domino}
  
\item {\it If in so doing  the bar is crossed add $1$ to $c$}

\end{itemize}
{\it If the  domino is a $[{s\atop a}]$
then clockwise from it find the first $[{b\atop a}]$
with $a<c$, call it ``$endsec$''
then cycle  clockwise from $endsec$ back to this $[{s\atop a}]$}
 
\begin{itemize}
\item {\it  for each encountered unmarked $[{b\atop a}]$
  add $1$ to $ndinv$ provided $a<c$ if the bar is not crossed  or $a<c+1$ 
  after the bar is crossed }
\end{itemize}
\sa

\centerline{ (the desired value of $ndinv$ is reached  after all the the small car dominos are marked)}

\vfill
 
 The successive configurations obtained after  the marking of
 small car dominos are displayed below with the same conventions 
 used on the right of \eqref{eq3p26}
\begin{equation}
  \Longrightarrow  \ess\ess
\vcenter{\hbox{ \includegraphics[width=2.2in]{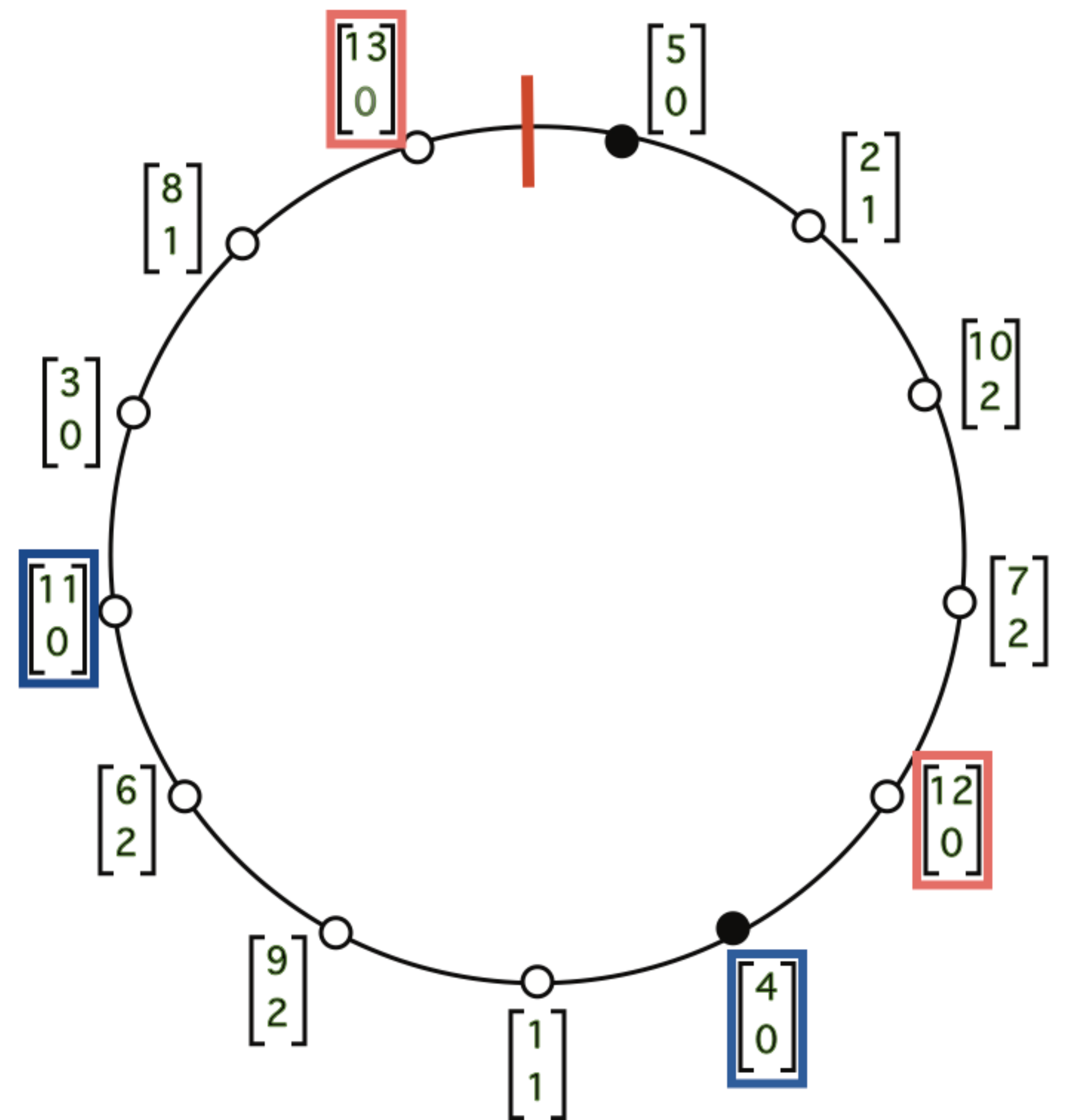}}} 
\ess\ess  \Longrightarrow  \ess\ess 
\vcenter{\hbox{ \includegraphics[width=2.2in]{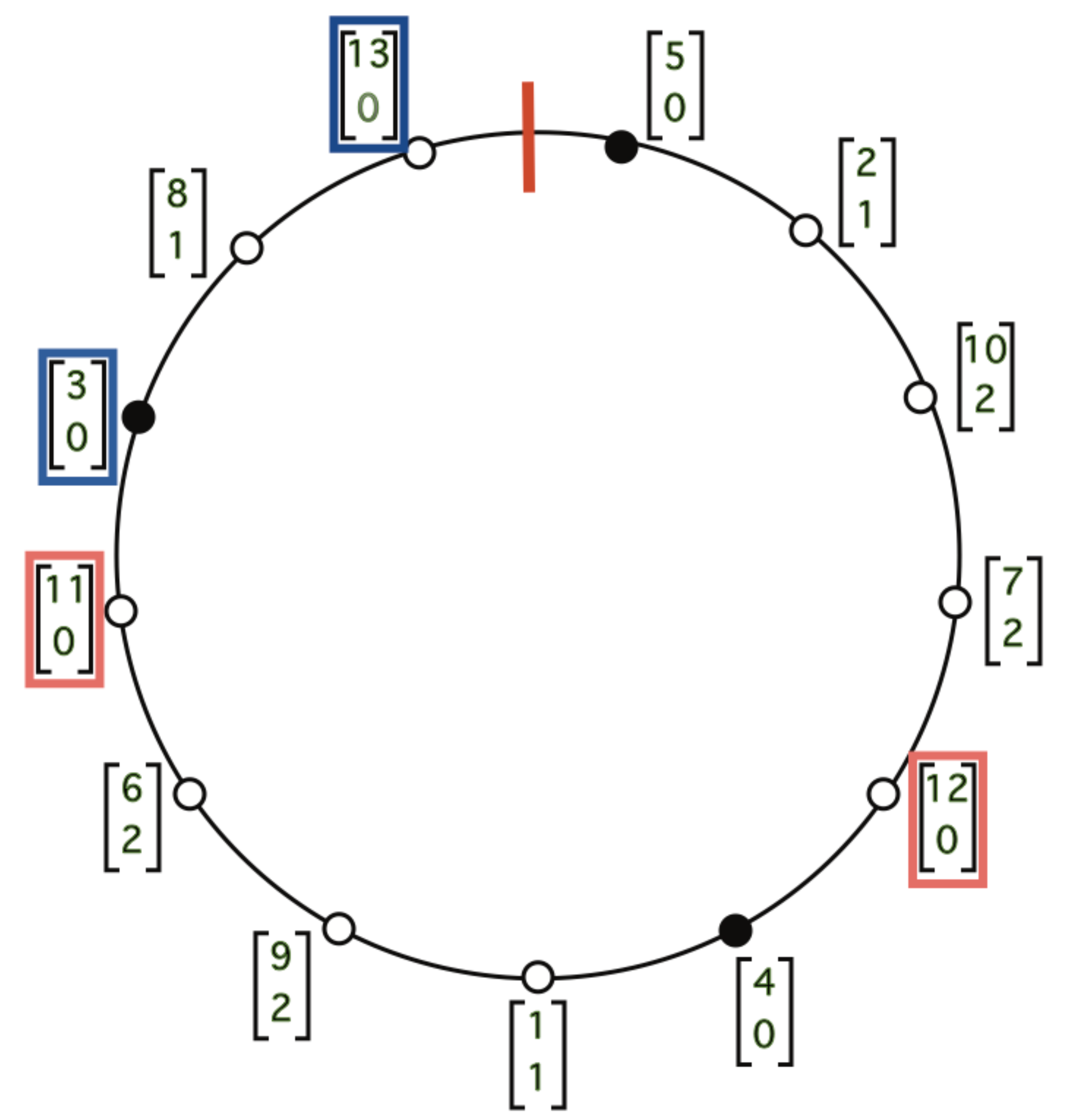}}} 
%\eqno 3.27
\label{eq3p27}
\end{equation}
\centerline { (at this point the $c$ value increases to $2$ and we obtain)}
  $$
$$     
\begin{equation}
  \Longrightarrow  \ess\ess
\vcenter{\hbox{ \includegraphics[width=2.2in]{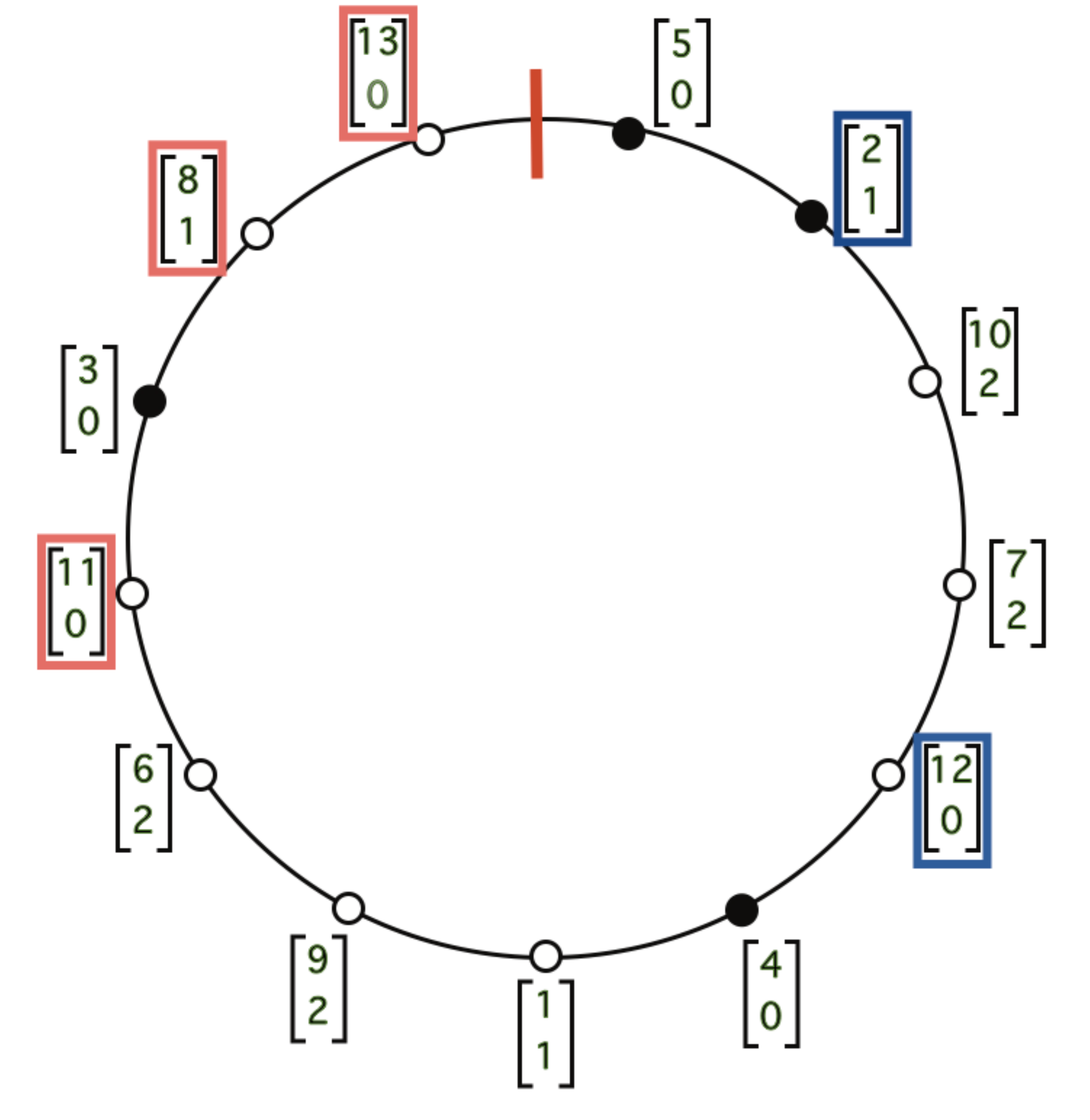}}} 
\ess\ess  \Longrightarrow  \ess\ess 
\vcenter{\hbox{ \includegraphics[width=2.2in]{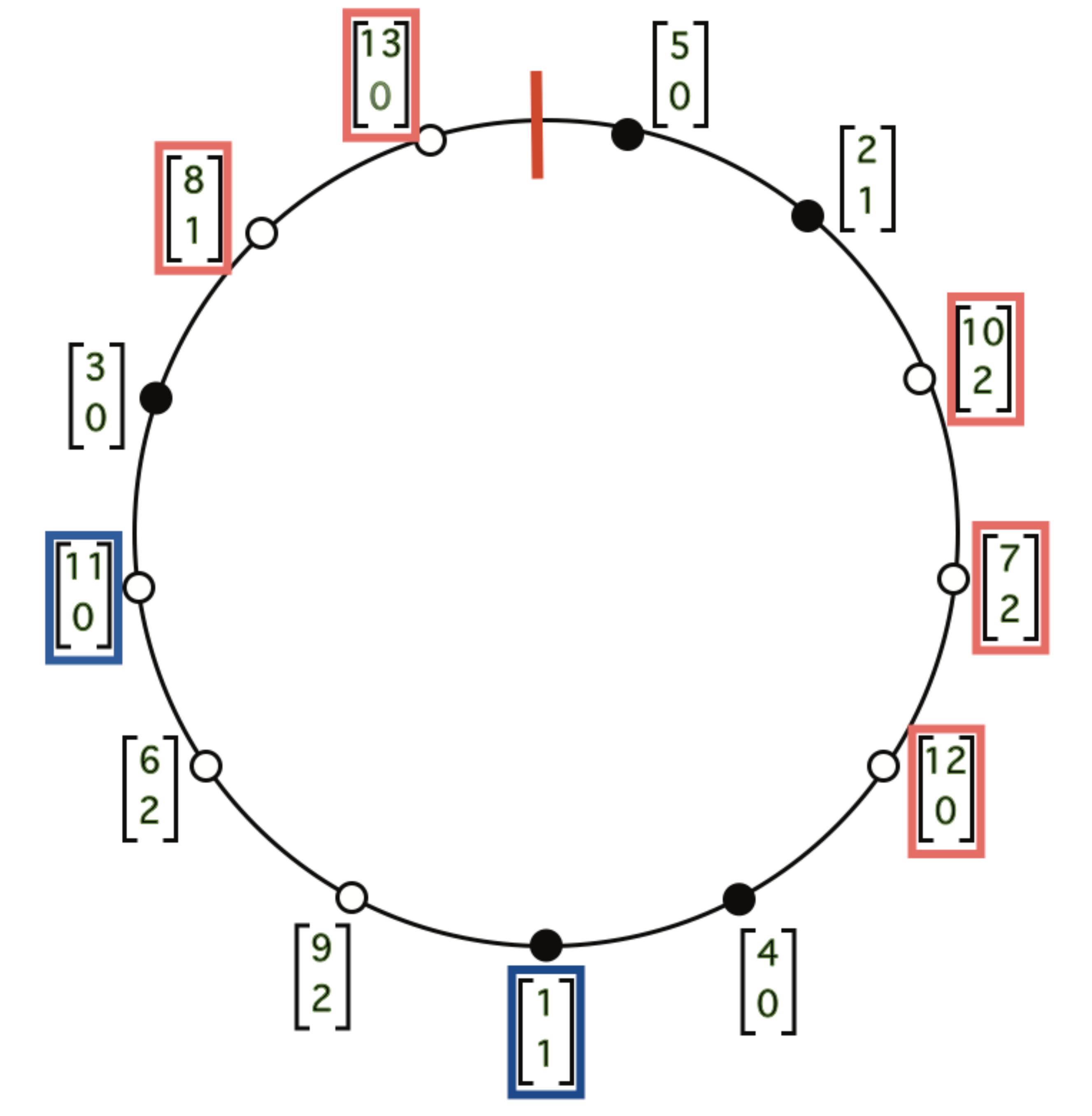}}} 
%\eqno 3.27
\label{eq3p27p5}
\end{equation}
    
\begin{center}(thus,   in this case, $ndinv=14$, which is the total number\\
  of lightly boxed  dominos in the previous five  configurations )
\end{center}
 
%Remark 3.2}
\begin{remark}
 We will not include a proof of the validity of this  second
 algorithm, since  A. S. Hicks
 and Y. Kim, using  their  discoveries, are able to provide
 in \cite{HicksKim} 
 a much simpler and more revealing validity argument
 than  we can offer with our present tools. Here it should  be sufficient  to
 acknowledge  that  the auxiliary 
 domino sequence resulting from Phase I together
 with the $c$ statistic  constructed in Phase II 
 have ultimately been put to such beautiful use 
 in subsequent work.
\end{remark}

Before  closing we should note that our $ndinv$  may
 have an   extension that can be used in a more general  settings than the present one. To see this, let us recall that
 the $2$ part case of the Shuffle Conjecture, proved by
 J. Haglund in \cite{Hag}, may be stated as follows
\begin{align}
\LL \nabla &e_n \scs h_jh_{n-j} \RR\ses \nonumber\\
&\sum_{PF \in\CAP_n} t^{area(PF)}q^{dinv(PF)}\chi(\sig(PF)\in 12\cdots j\sch j+1\cdots n)
%\eqno 3.28
\label{eq3p28}
\end{align}

Now replacing $n$ by $n+1-j$ and $J$ by $j$ in \eqref{eqIIp17},
for $(p_1,p_2,\ldots ,p_k)\models n+1-j$
we get
\begin{align}
\LL \Delta_{h_j} &\BC_{p_1}\BC_{p_2}\cdots \BC_{p_k}\, 1 \scs e_{n+1-j} \RR
\ses \nonumber\\
& \sum_{\multi{PF\in \CAP_{n+1}(k) 
\cr p(big(PF))=(p_1,p_2,\ldots ,p_k)}}
t^{area(PF)}q^{ndinv(PF)}\chi(\sig(PF)\in 12\cdots j\sch j+1\cdots n+1)
%\eqno 3.29
\label{eq3p29}
\end{align}
This given,  since is was shown in \cite{HMZ} that we may write
$$
e_{n+1}\ses \sum_{p\models n+1}
\BC_{p_1}\BC_{p_2}\cdots \BC_{p_{l(p)}}1
$$
it follows,  by summing \eqref{eq3p29} over all compositions of $n+1$,
that we also have
\begin{equation}
\LL \Delta_{h_j} e_{n+1-j} \scs e_{n+1-j} \RR
\ses  \sum_{PF\in \CAP_{n+1} }\hskip - .28 in\ ^{(*)} \ssp 
t^{area(PF)}q^{ndinv(PF)}\chi(\sig(PF)\in 12\cdots j\sch j+1\cdots n+1)
%\eqno 3.30
\label{eq3p30}
\end{equation}
where the ``$(*)$'' is to signify that the sum is over all parking functions in the $n+1\times n+1$ lattice square which have
 the biggest car $n+1$ in cell $(1,1)$. 
 But it was also shown in \cite{Hag}\ that we have
 $$
 \LL\Delta_{h_j}e_{n+1-j}\scs e_{n+1-j}\RR
 \ses \LL \nabla e_{n} 
\scs h_jh_{n-j} \RR
 $$ 
Thus \eqref{eq3p30} may also be rewritten in the form
\begin{equation}
\LL \nabla e_{n}\scs h_jh_{n-j}\RR
\ses  \sum_{PF\in \CAP_{n+1} }\hskip - .28 in\ ^{(*)} \ssp 
t^{area(PF)}q^{ndinv(PF)}\chi(\sig(PF)\in 12\cdots j\sch j+1\cdots n+1)
%\eqno 3.31
\label{eq3p31}
\end{equation}
which gives another parking function interpretation
to this remarkable polynomial. It is natural then to
ask if this kind of result involving the same $ndinv$, or a
suitable extension of it, may give a new parking function interpretation to any of the polynomials occurring 
on the left hand side of \eqref{eqIIp7}. If that were the case
then that would provide an alternate form of the 
Shuffle Conjecture. It is interesting to note that 
computer exploration has led us to conjecture that for
$p=(p_1,p_2,\ldots ,p_k)\models n$
the polynomials
$$
\LL \DD_{h_{J_1}e_{J_2}} \BC_{p_1} \BC_{p_1}\cdots  \BC_{p_1} 1\scs e_n\RR
$$
have non negative integer coefficients. This yields 
us yet  another avenue by which the results of this
paper can be extended. It should be worthwhile 
to pursue these avenues in further investigations 
on the connections between Parking  Functions and
the Theory of Macdonald Polynomials.
\end{section}

\end{document}